\documentclass[11pt, a4paper]{amsart}

\usepackage{amssymb,amsmath,amsthm}
\usepackage{tikz-cd}
\usetikzlibrary{arrows}
\usepackage{fullpage}
\usepackage{graphicx}
\usepackage[hidelinks]{hyperref}
\usepackage{mathtools}
\usepackage[titletoc,title]{appendix}
\usepackage{graphicx}
\usepackage{caption}
\usepackage{subcaption}
\usepackage{blindtext}
\usepackage[titletoc,title]{appendix}
\usepackage[english]{babel}
\usepackage{mathrsfs}
\usepackage{cite}
\usepackage{mathabx}
\usepackage{xcolor}

\newtheorem{theorem}{Theorem}[section]
\newtheorem{lemma}[theorem]{Lemma}
\newtheorem{corollary}[theorem]{Corollary}
\newtheorem{proposition}[theorem]{Proposition}
\theoremstyle{definition}
\newtheorem{definition}[theorem]{Definition}
\newtheorem{example}[theorem]{Example}
\newtheorem{problem}[theorem]{Problem}
\newtheorem{remark}[theorem]{Remark}

\newtheorem{question}[theorem]{Question}

\newtheorem*{theorem*}{Theorem}
\newtheorem*{conjecture*}{Conjecture}

\newcommand{\isom}{\cong}

\newcommand{\N}{\mathbb{N}}
\newcommand{\Z}{\mathbb{Z}}

\newcommand{\R}{\mathbb{R}}
\makeatletter
\def\Ddots{\mathinner{\mkern1mu\raise\p@
\vbox{\kern7\p@\hbox{.}}\mkern2mu
\raise4\p@\hbox{.}\mkern2mu\raise7\p@\hbox{.}\mkern1mu}}

\makeatother

\DeclareMathOperator{\Ima}{Im}

\newcommand{\normal}[1]{\langle\langle #1 \rangle\rangle}

\def\immerses{\looparrowright}
\def\injects{\hookrightarrow}

\DeclareSymbolFontAlphabet{\amsmathbb}{AMSb}

\DeclareMathOperator{\rk}{rk}
\DeclareMathOperator{\rr}{rr}
\DeclareMathOperator{\pr}{pr}

\DeclareMathOperator{\core}{Core}

\DeclareMathOperator{\aut}{Aut}

\DeclareMathOperator{\stab}{Stab}

\DeclareMathOperator{\deck}{Deck}
\DeclareMathOperator{\zd}{\mathbb{Z}D}
\DeclareMathOperator{\ab}{ab}
\DeclareMathOperator{\bs}{BS}

\DeclarePairedDelimiter\abs{\lvert}{\rvert}

\makeatletter
\let\oldabs\abs
\def\abs{\@ifstar{\oldabs}{\oldabs*}}

\newcounter{cases}
\newcounter{subcases}[cases]
\newenvironment{mycase}
{
    \setcounter{cases}{0}
    \setcounter{subcases}{0}
    \newcommand{\case}
    {
        \par\indent\stepcounter{cases}\textbf{Case \thecases.}
    }
    
}
{
    \par
}
\renewcommand*\thecases{\arabic{cases}}

\newcommand{\fakeenv}{} %%% prints the emptystring

\newenvironment{restate}[2]  %%% restate takes two arguments 
{ 
 \renewcommand{\fakeenv}{#2} %%% So now \fakeenv prints #2
 \theoremstyle{plain} 
 \newtheorem*{\fakeenv}{#1~\ref{#2}} %%% so now #2 is the name of a
 \begin{\fakeenv}
}
{
 \end{\fakeenv}
}
\begin{document}

\title{One-relator hierarchies}
\author{Marco Linton}

%\date{\today}

\address{University of Oxford, Oxford, OX2 6GG, UK}

\email{marco.linton@maths.ox.ac.uk}

\begin{abstract}
We prove that one-relator groups with negative immersions are hyperbolic and virtually special; this resolves a recent conjecture of Louder and Wilton. As a consequence, one-relator groups with negative immersions are residually finite, linear and have isomorphism problem decidable among one-relator groups. Using the fact that parafree one-relator groups have negative immersions, we answer a question of Baumslag's from 1986. The main new tool we develop is a refinement of the classic Magnus--Moldavanskii hierarchy for one-relator groups. We introduce the notions of $\Z$-stable HNN-extensions and $\Z$-stable hierarchies. We then show that a one-relator group is hyperbolic and has a quasi-convex one-relator hierarchy if and only if it does not contain a Baumslag--Solitar subgroup and has a $\Z$-stable one-relator hierarchy.
\end{abstract}

\maketitle

\section{Introduction}

One-relator groups, despite their simple definition, have been stubbornly resistant to geometric characterisations. A well known conjecture attributed to Gersten states that a one-relator group is hyperbolic if and only if it does not contain Baumslag--Solitar subgroups \cite{gersten_92,allcock_99}. Modifying the hypothesis or conclusion of this conjecture appears to be problematic. For example, conjectures attempting to classify one-relator groups that are automatic \cite{gersten_92,myasnikov_11} or act freely on a CAT(0) cube complex \cite{wise_14} have been posed, but both have been disproven \cite{gardam_18}. Conjectures attempting to classify when (subgroups of) groups of finite type are hyperbolic have also been posed \cite{brady_99,bestvina_04}, but have also been disproven \cite{italiano_21}. More variations can be found, for example, in \cite{gromov_93,wise_05_cat,gardam_21}.

Recently, a different type of geometric characterisation for one-relator groups has emerged from work of Helfer and Wise \cite{helfer_16} and, independently, Louder and Wilton \cite{louder_17}. Namely, if $X$ is the presentation complex of a one-relator group, then:
\begin{enumerate}
\item $X$ has not-too-positive immersions \cite{louder_17}.
\item $X$ has non-positive immersions if and only if $\pi_1(X)$ is torsion-free \cite{helfer_16}.
\item $X$ has negative immersions if and only if every two-generator subgroup of $\pi_1(X)$ is free \cite{louder_21}.
\end{enumerate}
Moreover, each of these properties is decidable from $X$ \cite{karrass_60,louder_21}. Since one-relator groups with negative immersions cannot contain Baumslag--Solitar subgroups, this led Louder and Wilton to make the following conjecture \cite[Conjecture 1.9]{louder_21}.

\begin{conjecture*}
\label{hyperbolic_conjecture}
Every one-relator group with negative immersions is hyperbolic.
\end{conjecture*}

The original notions of non-positive and negative immersions are due to Wise \cite{wise_03,wise_04}; note that they are not the same as Louder and Wilton's \cite{louder_21_uniform}. The aforementioned conjecture can be considered as a special case of an older conjecture of Wise \cite[Conjecture 14.2]{wise_04}. Another related property is that of uniform negative immersions, introduced in \cite{louder_21_uniform}. Results analogous to those proved by Louder and Wilton \cite{louder_21,louder_21_uniform} for one-relator groups with negative immersions have also been shown for fundamental groups of two-complexes with a stronger form of uniform negative immersions by Wise \cite{wise_21}. However, having uniform negative immersions turns out to be equivalent to having negative immersions in the case of one-relator complexes \cite[Theorem C]{louder_21_uniform}. 

Louder and Wilton's conjecture has been experimentally verified for all one-relator groups with negative immersions that admit a one-relator presentation with relator of length less than 17 \cite{cashen_21}. In this article we verify their conjecture and in fact prove more.

\begin{restate}{Theorem}{hyperbolicity_theorem}
One-relator groups with negative immersions are hyperbolic and virtually special.
\end{restate}

Despite the maturity of the theory of one-relator groups, the isomorphism problem has remained almost untouched. A subclass of one-relator groups that is often mentioned when demonstrating the difficulty of this problem is that of parafree one-relator groups \cite{chandler_82,baumslag_19}. Indeed, Baumslag asked \cite[Problem 4]{baumslag_85} whether the isomorphism problem for parafree one-relator groups is solvable. See also \cite[Question 3]{fine_21}. A large body of work has been carried out on distinguishing parafree one-relator groups; see, for example, \cite{fine_97,hung_17,hung_20,chen_21} and \cite{lewis_94,baumslag_04} for computational experiments. For several families of examples of parafree one-relator groups, see also \cite{baumslag_06}. A consequence of Theorem \ref{hyperbolicity_theorem} is the following.

\begin{restate}{Corollary}{isomorphism}
The isomorphism problem for one-relator groups with negative immersions is decidable within the class of one-relator groups.
\end{restate}

By showing that parafree one-relator groups have negative immersions, we may answer Baumslag's question in the affirmative.

\begin{restate}{Corollary}{parafree}
The isomorphism problem for parafree one-relator groups is decidable within the class of one-relator groups.
\end{restate}

We conclude this section by mentioning a couple of other corollaries.

\begin{restate}{Corollary}{residual_finiteness}
One-relator groups with negative immersions are residually finite and linear.
\end{restate}

\begin{restate}{Corollary}{hyperbolicity_corollary}
Every finitely generated subgroup of a one-relator group with negative immersions is hyperbolic.
\end{restate}

We now present the main new theorems that go into proving Theorem \ref{hyperbolicity_theorem}.

\subsection{Magnus--Moldavanskii--Masters hierarchies}

First conceived by Magnus in his thesis \cite{magnus_30}, the Magnus hierarchy is possibly the oldest general tool in the theory of one-relator groups. After the introduction of the theory of HNN-extensions of groups, the hierarchy was later refined to be called the Magnus--Moldavanskii hierarchy \cite{moldavanskii_67}: if $G$ is a one-relator group, then there is a diagram of monomorphisms of one-relator groups
\[
\begin{tikzcd}
G = G_0 \arrow[r, hook]		   & G'_0 \\
G_1	 \arrow[ur, hook] \arrow[r, hook] & G'_1 \\
\cdots \arrow[ur, hook] \arrow[r, hook] & \cdots \\
G_N \arrow[ur, hook]			& 
\end{tikzcd}
\]
such that $G_i'\isom G_{i+1}*_{\psi_i}$ where $\psi_i$ identifies two Magnus subgroups of $G_{i+1}$, and $G_N$ splits as a free product of cyclic groups. The proofs of many results for one-relator groups then proceed by induction on the length of such a hierarchy. See \cite{magnus_04} for a classically flavoured introduction to one-relator groups with many such examples.

In \cite{masters_06}, Masters showed that we can dispense with the horizontal homomorphisms. In other words, if $G$ is a one-relator group, there is a sequence of monomorphisms of one-relator groups:
\[
G_N\injects...\injects G_1\injects G_0 = G
\]
such that $G_i \isom G_{i+1}*_{\psi_i}$ where $\psi_i$ identifies two Magnus subgroups of $G_{i+1}$, and $G_N$ splits as a free product of cyclic groups.

\subsection{One-relator hierarchies} 

The versatility of the Magnus--Moldavanskii hierarchy comes from the fact that it may be described very explicitly in terms of one-relator presentations. Masters' hierarchy is conceptually simpler, but is not so explicit. By working with two-complexes, we may reconcile both of these advantages. Our version of the hierarchy can be stated as follows.

\begin{restate}{Theorem}{new_hierarchy}
Let $X$ be a finite one-relator complex. There exists a finite sequence of immersions of one-relator complexes:
\[
X_N\immerses ...\immerses X_1\immerses X_0 = X
\]
such that $\pi_1(X_i)\isom \pi_1(X_{i+1})*_{\psi_i}$ where $\psi_i$ is induced by an identification of Magnus subgraphs, and such that $\pi_1(X_N)$ is finite cyclic.
\end{restate}

Each such HNN-splitting is called a \emph{one-relator splitting} and we call such a sequence of immersions a \emph{one-relator tower} for $X$. The sequence of immersions are \emph{tower maps} as defined in \cite{howie_81}: each immersion factors as an inclusion composed with a cyclic cover. The homomorphisms $\psi_i$ are thus induced by an identification of two subcomplexes by the deck group action. A maximal one-relator hierarchy is a maximal tower lifting of the induced map of the closed two-cell to the one-relator complex. Note that not every such tower lifting will induce a hierarchy of HNN-extensions of one-relator groups and thus the novelty of Theorem \ref{new_hierarchy} lies in showing that such a tower always exists. The proof of Theorem \ref{tower} relies on a technical and combinatorial analysis of cyclic covers of two-complexes. Some of its applications not mentioned in this article are explored in the author's thesis \cite{my_thesis}.

\subsection{$\Z$-stable one-relator hierarchies}

We now introduce the notion of $\Z$-stable one-relator hierarchies. Let $H$ be a group and $\psi:A\to B$ an isomorphism between subgroups of $H$. Inductively define 
\[
\mathcal{A}^{\psi}_0 = \{[H]\}, \quad \ldots ,\quad \mathcal{A}^{\psi}_{i+1} = \{[\psi(A\cap A_i)]\}_{A_i\in [A_i]\in \mathcal{A}^{\psi}_i}, \quad \ldots
\]
where $[A_i]$ denotes the conjugacy class of $A_i$ in $H$. Then we denote by $\bar{\mathcal{A}}_i^{\psi}\subset \mathcal{A}_i^{\psi}$ the subset corresponding to the conjugacy classes of non-cyclic subgroups. Define the \emph{$\Z$-stable number} $s\Z(\psi)$ of $\psi$ to be
\[
s\Z(\psi) = \sup{\{k + 1 \mid \bar{\mathcal{A}}_k^{\psi}\neq\emptyset\}}\in \N\cup\{\infty\}
\]
where $s\Z(\psi) = \infty$ if $\bar{\mathcal{A}}_i^{\psi}\neq \emptyset$ for all $i$. In general, even $\bar{\mathcal{A}}_2^{\psi}$ may contain infinitely many conjugacy classes of subgroups. However, we show in Lemma \ref{stab_ranks} that if $\pi_1(X)\isom \pi_1(X_1)*_{\psi}$ is a one-relator splitting as in Theorem \ref{new_hierarchy}, then
\[
\sum_{[A_n]\in \bar{\mathcal{A}}_n^{\psi}}\rr(A_n)\leq \rr(A)
\]
where $\rr(A) = \max{\{0, \rk(A) - 1\}}$ denotes the reduced rank of $A$.

\begin{definition}
A one-relator hierarchy $X_N\immerses...\immerses X_1\immerses X_0 = X$ is a \emph{$\Z$-stable hierarchy} if $s\Z(\psi_i)<\infty$ for all $i<N$.
\end{definition}

Our next result establishes an equivalence between quasi-convex one-relator hierarchies and $\Z$-stable one-relator hierarchies of hyperbolic one-relator groups.

\begin{restate}{Theorem}{main}
Let $X$ be a one-relator complex and $X_N\immerses ...\immerses X_1\immerses X_0 = X$ a one-relator hierarchy. The following are equivalent:
\begin{enumerate}
\item\label{itm:hqch} $X_N\immerses ...\immerses X_1\immerses X_0 = X$ is a quasi-convex hierarchy and $\pi_1(X)$ is hyperbolic.
\item\label{itm:acylindrical_h} $X_N\immerses ...\immerses X_1\immerses X_0 = X$ is an acylindrical hierarchy.
\item\label{itm:bs} $X_N\immerses ...\immerses X_1\immerses X_0 = X$ is a $\Z$-stable hierarchy and $\pi_1(X)$ contains no Baumslag--Solitar subgroups.
\end{enumerate}
Moreover, if any of the above is satisfied, then $\pi_1(X)$ is virtually special and the image of $\pi_1(A)$ in $\pi_1(X)$ is quasi-convex for any connected subcomplex $A\subset X_i$.
\end{restate}

In \cite{wise_21_quasiconvex}, Wise shows that Magnus--Moldavanskii hierarchies of one-relator groups with torsion are quasi-convex. We show that all one-relator hierarchies of one-relator complexes $X$ satisfying either of the following are $\Z$-stable:
\begin{enumerate}
\item $\pi_1(X)$ has torsion (by Corollary \ref{torsion_stable})
\item $X$ has negative immersions (by Corollary \ref{stable_negative})
\end{enumerate}
In the first case, since one-relator groups with torsion are hyperbolic \cite{newman_68}, we recover Wise's result. In the second case, since fundamental groups of one-relator complexes with negative immersions are two-free, they cannot contain Baumslag--Solitar subgroups and so we prove Louder and Wilton's conjecture. 

By showing that hyperbolic groups with quasi-convex hierarchies are virtually special \cite{wise_21_quasiconvex}, Wise also proved that one-relator groups with torsion are residually finite, settling an old conjecture of Baumslag \cite{baumslag_67}. As a consequence of Wise's work and Theorem \ref{main}, we may also establish virtual specialness, residual finiteness and linearity for one-relator groups with negative immersions.

Theorem \ref{main} is the crux of the article and is easily applicable to concrete examples, as demonstrated by the following example.

\begin{example}
\label{stable_example}
Consider the following one-relator hierarchy of length one:
\[
\langle x, y, z \mid z^2yz^2x^{-2} \rangle*_{\psi} \isom \langle a, b \mid b^2a^2b^{-1}aba^2b^{-2}a^{-2}\rangle
\]
where $\psi$ is given by $\psi(x) = y$, $\psi(y) = z$. Since $y = z^{-2}x^2z^{-2}$, we see that $\langle x, y, z\rangle$ is a free group freely generated by $x$ and $z$. Thus we have:
\begin{align*}
\bar{\mathcal{A}}_0^{\psi} &= \{[\langle x, z\rangle]\}\\
\bar{\mathcal{A}}_1^{\psi} &= \{[\langle x^2, z\rangle]\}\\
\bar{\mathcal{A}}_2^{\psi} & = \{[\langle (z^{-2}x^2z^{-2})^2, z\rangle]\}\\
\bar{\mathcal{A}}_3^{\psi} & = \emptyset
\end{align*}
Hence, $s\Z(\psi) = 3$ and this hierarchy is $\Z$-stable. Using a criterion for finding Baumslag--Solitar subgroups we develop in Subection \ref{sec:BS_criterion}, in Example \ref{no_BS_example} we show that 
\[
\langle a, b \mid b^2a^2b^{-1}aba^2b^{-2}a^{-2}\rangle
\]
does not contain a Baumslag--Solitar subgroup. Thus, by Theorem \ref{main}, it is hyperbolic and virtually special.
\end{example}

The choice of group in Example \ref{stable_example} is motivated by \cite{cashen_21} in which the authors verify hyperbolicity for one-relator groups with relator of length less than 17, using a combination of results from the literature and software in GAP and kbmag. This particular example does not satisfy any of the criteria the authors used: it is torsion-free, it is not small cancellation, it is not cyclically or conjugacy pinched, it does not satisfy the hypotheses of \cite[Theorems 3 \& 4]{ivanov_98} nor those of the hyperbolicity criterion in \cite{blufstein_19}. Moreover, it does not have negative immersions and does not split as an HNN-extension of a free group with a free factor edge group so that the results from \cite{mutanguha_21,mutanguha_21_stable} also do not apply.

\subsection{Outline of the article}

In Section \ref{sec:prelim_1}, we introduce the necessary background and terminology on graphs and one-relator complexes. There we introduce the notion of a strongly inert graph immersion and use it to prove a key result, Theorem \ref{subgroup_Magnus_intersection}, bounding the sum of reduced ranks of intersections of certain subgroups in a one-relator group. This feeds into the proof of Theorem \ref{main}. In Section \ref{sec:prelim_2}, we cover graphs of spaces and Bass--Serre theory. There we prove Proposition \ref{homotopy_equivalence}, a useful tool which will allow us to find HNN-splittings from cyclic covers of one-relator complexes as in Theorem \ref{new_hierarchy}. In Section \ref{sec:zdomain} we introduce the notion of a $\Z$-domain, prove that $\Z$-domains of $\Z$-covers of finite CW-complexes exist and that minimal $\Z$-domains of $\Z$-covers of one-relator complexes are one-relator complexes. Combined with a complexity reduction argument, we then use this to establish Theorem \ref{new_hierarchy}. Section \ref{sec:normal_forms} is dedicated to showing that a hyperbolic one-relator group with a quasi-convex hierarchy has all of its Magnus subgroups quasi-convex. The proof relies on a careful analysis of normal forms induced by the hierarchy from Theorem \ref{new_hierarchy}. At this point of the article, it is possible to establish the equivalence of (\ref{itm:hqch}) and (\ref{itm:acylindrical_h}) from Theorem \ref{main}. In Section \ref{sec:stable}, $\Z$-stable hierarchies are introduced and some of their properties are proven. All hierarchies of one-relator groups with torsion and negative immersions are then shown to have such hierarchies. The section is concluded with Theorem \ref{acylindrical_equivalent} which shows the equivalence of (\ref{itm:acylindrical_h}) and (\ref{itm:bs}) from Theorem \ref{main} and provides a criterion for finding Baumslag--Solitar subgroups. Finally, Section \ref{sec:main} is dedicated to combining all the new tools from the article to prove our main results, Theorems \ref{main} and \ref{hyperbolicity_theorem}.

\subsection*{Acknowledgements} We would like to thank Saul Schleimer and Henry Wilton for many stimulating conversations that helped improve the exposition of this article. We would also like to thank Lars Louder for his invaluable comments on Lemma \ref{strongly_inert}. Finally, we would like to thank the anonymous referee for the many detailed and insightful comments which have enormously improved this article.

\section{Graphs and one-relator complexes}
\label{sec:prelim_1}

A \emph{graph} $\Gamma$ is a 1-dimensional CW-complex. We will write $V(\Gamma)$ for the collection of 0-cells or \emph{vertices} and $E(\Gamma)$ for the collection of 1-cells or \emph{edges}. We will usually assume $\Gamma$ to be oriented. An orientation will be induced by maps $o:E(\Gamma)\to V(\Gamma)$ and $t:E(\Gamma)\to V(\Gamma)$, the \emph{origin} and \emph{target} maps. For simplicity, we will write $I$ to denote any connected graph whose vertices all have degree two, except for two vertices of degree one. Then $S^1$ will denote a connected graph all of whose vertices have degree precisely two.

A map between graphs $f:\Gamma\to \Gamma'$ is \textit{combinatorial} if it sends each vertex to a vertex and each edge homeomorphically to an edge. A combinatorial map is an \textit{immersion} if it is also locally injective. Combinatorial graph maps $\lambda:I\to \Gamma$, $\lambda:S^1\to \Gamma$ will be called \emph{paths} and \emph{cycles} respectively. The \emph{length} of a combinatorial path $\lambda:I\to \Gamma$ is the number of edges in $I$ and is denoted by $\abs{\lambda}$. If $\lambda:I\to X$ is a path, we may often identify the vertices of $I$ with the integers $0, 1, ..., \abs{\lambda}$ so that $\lambda(i)$ is the $i^{\text{th}}$ vertex that $\lambda$ traverses. We also put $o(\lambda) = \lambda(0)$ and $t(\lambda) = \lambda(\abs{\lambda})$. We similarly define the length of a cycle $\lambda:S^1\to \Gamma$. A cycle $\lambda:S^1\to \Gamma$ is \emph{primitive} if $\lambda$ does not factor through any non-trivial covering map $S^1\to S^1$. We will call it \emph{imprimitive} otherwise. We will write $\deg(\lambda)$ to denote the maximal degree of a covering map $S^1\to S^1$ that $\lambda$ factors through. Note that $\deg(\lambda) = 1$ if and only if $\lambda$ is primitive. We remark that this definition of primitivity should be compared with the definition of primitivity from the theory of combinatorics of words, not with the definition of primitivity in the theory of free groups. In particular, $\lambda$ being imprimitive is not the same as $\Ima(\lambda_*)$ being imprimitive in $\pi_1(\Gamma)$.

The \emph{core} of a graph $\Gamma$ is the subgraph consisting of the union of all the images of immersed cycles $S^1\immerses \Gamma$ and will be denoted by $\core(\Gamma)$. Note that if $\Gamma$ is a forest, then $\core(\Gamma) = \emptyset$. In particular $\core(\Gamma)$ is unique.

\subsection{Strongly inert graph immersions}

If $G$ is a group and $g, h\in G$, we will adopt the usual convention that $h^g = g^{-1}\cdot h\cdot g$. If $H<G$, we will write $[H]$ to denote the conjugacy class of $H$ in $G$. More generally, If $X, Y\subset G$ are subsets, then define the \emph{$Y$-conjugacy class} of $X$ to be the following:
\[
[X]_Y = \{X^y\mid y\in Y\}
\]

A subgroup $H<G$ is called \emph{inert} if for every subgroup $K<G$, we have $\rk(H\cap K)\leq \rk(K)$. This definition was first introduced in \cite{dicks_96}, motivated by the study of fixed subgroups of endomorphisms of free groups. More generally, as defined in \cite{ivanov_18}, we say that $H$ is \emph{strongly inert} if for every subgroup $K<G$, we have:
\[
\sum_{KgH}\rr(H\cap K^g)\leq \rr(K)
\]
Examples of strongly inert subgroups of free groups are:
\begin{enumerate}
\item subgroups of rank at most two \cite{tardos_92},
\item subgroups that are the fixed subgroup of an injective endomorphism \cite[Theorem IV.5.5]{dicks_96},
\item subgroups that are images of immersions of free groups, as defined in \cite{kapovich_00}.
\end{enumerate}
The latter example follows by observing that the fibre product (in the sense of \cite{sta_83}) of two rose graphs can contain at most a single vertex of valence greater than two.

\begin{remark}
\label{cyclonormal_remark}
If $H<G$ is strongly inert, then applying the definition to $K = H$, we get that $\rr(H\cap H^g)\leq 0$ for all $g\notin H$. Thus, $H$ is \emph{cyclonormal}.
\end{remark}

Translating this to graphs, we make the following definition.

\begin{definition}
A graph immersion $\gamma:\Gamma\immerses \Delta$ is \emph{strongly inert} if for all graph immersions $\lambda:\Lambda\immerses\Delta$, the following is satisfied:
\[
\chi(\core(\Gamma\times_{\Delta}\Lambda))\geq \chi(\core(\Lambda)).
\]
\end{definition}

Let us briefly explain how this is a translation of strong inertia to graphs. If $\Delta$ is a (non-empty) connected core graph, we have that $\rr(\pi_1(\Delta)) = -\chi(\Delta)$. By work of Stallings \cite{sta_83}, if $\gamma\colon\Gamma\immerses\Delta$ and $\lambda\colon\Lambda\immerses\Delta$ are graph immersions, then the components of $\core(\Gamma\times_{\Delta}\Lambda)$ are in natural bijection with the non-trivial intersections $\gamma_*\pi_1(\Gamma)\cap\lambda_*\pi_1(\Lambda)^g$, as $g$ varies over representatives for the double cosets $\lambda_*\pi_1(\Lambda)g\gamma_*\pi_1(\Gamma)$. As such, we have 
\[
-\chi(\core(\Gamma\times_{\Delta}\Lambda)) = \sum_{\lambda_*\pi_1(\Lambda)g\gamma_*\pi_1(\Gamma)}\rr(\gamma_*\pi_1(\Gamma)\cap\lambda_*\pi_1(\Lambda)^g).
\]

The following lemma will produce more examples of strongly inert subgroups of free groups.

\begin{lemma}
\label{strongly_inert}
Let $\gamma:\Gamma\immerses \Delta$ be an immersion of finite graphs. If $\gamma_*\pi_1(\Gamma)*F$ is an inert subgroup of $\pi_1(\Delta)*F$ for all free groups $F$, then $\gamma$ is a strongly inert graph immersion. 
\end{lemma}

\begin{proof}
Suppose for a contradiction that there exists a graph immersion $\lambda:\Lambda\immerses \Delta$ such that:
\[
\chi(\core(\Gamma\times_{\Delta}\Lambda))< \chi(\core(\Lambda)).
\]
Let $\Theta_0, ..., \Theta_k\subset \core(\Gamma\times_{\Delta}\Lambda)$ be the connected components. Choose vertices, $(v_i, w_i)\in V(\Theta_i)$ for each $0\leqslant i\leqslant k$. Let $\Delta'$ be the graph with vertex set $V(\Delta)$ and edge set $E(\Delta)\sqcup \{e_1, \ldots, e_k\}$ where the origin and target of each $e_i$ is the basepoint. Similarly, we let $\Gamma'$ be the graph with vertex set $V(\Gamma)$ and edge set $E(\Gamma)\sqcup \{f_1, \ldots, f_k\}$ where the origin and target of each $f_i$ is the basepoint. We then define a graph map $\gamma'\colon \Gamma'\immerses\Delta'$ by $\gamma'(v) = \gamma(v)$ for each $v\in V(\Gamma')$, $\gamma'(f) = \gamma(f)$ if $f\in E(\Gamma)$ and $\gamma'(f_i) = e_i$ for each $1\leqslant i\leqslant k$. We have that $\pi_1(\Delta') = \pi_1(\Delta)*F$ where $F$ is a free group of rank $k$ and $\gamma'_*\pi_1(\Gamma') = \gamma_*\pi_1(\Gamma)*F$. Hence, by assumption, $\gamma'_*\pi_1(\Gamma')$ is inert in $\pi_1(\Delta')$.

Now for each $1\leqslant i\leqslant k$, let $g_i:I_i\immerses \Gamma$ be any immersed path such that $g_i$ begins at $v_{i-1}$, ends at $v_{i}$, traverses $f_i$ precisely once and does not traverse $f_j$ for any $j\neq i$. Let $\Lambda'$ be the graph obtained from $\Lambda$ by attaching segments $I_i$ with origin $w_{i-1}$ and with target $w_{i}$ for each $1\leqslant i\leqslant k$. Let $\lambda':\Lambda'\to \Delta'$ be the map obtained by extending $\lambda$ by defining $\lambda'\mid_{I_i} = \gamma\circ g_i$. Finally, let $\lambda'':\Lambda''\immerses\Delta$ be the graph immersion obtained from $\lambda'$ by folding. By construction, for each $i$ there is only one edge in $\Lambda'$ that maps to $e_i$. Thus, each such edge does not get identified with any other edge under the folding map and so the folding map $\Lambda'\to \Lambda''$ is a homotopy equivalence, restricting to an isomorphism on $\Lambda\subset\Lambda'$. Now not only will $\core(\Gamma'\times_{\Delta'}\Lambda'')$ contain $\core(\Gamma\times_{\Delta}\Lambda)$ as a subgraph, but it will also contain segments connecting $\Theta_{i-1}$ to $\Theta_i$ for all $1\leqslant i\leqslant k$ by construction. We have $\chi(\Lambda'') = \chi(\Lambda) - k$, and so:
\begin{align*}
\chi(\core(\Gamma'\times_{\Delta'}\Lambda'')) &= \chi(\core(\Gamma'\times_{\Delta'}\Lambda)) - k\\
								  &<\chi(\core(\Lambda'')).
\end{align*}
However, $\core(\Gamma'\times_{\Delta'}\Lambda'')$ is connected by construction, contradicting the fact that $\gamma'_*\pi_1(\Gamma')$ was inert in $\pi_1(\Delta')$.
\end{proof}

\subsection{One-relator complexes and Magnus subgraphs}

For simplicity, we will mostly be restricting our attention to particular kinds of CW-complexes called combinatorial $2$-complexes.

\begin{definition}
A \emph{combinatorial $2$-complex} $X$ is a $2$-dimensional CW-complex whose attaching maps are all immersions. We will usually write $X = (\Gamma, \lambda)$ where $\Gamma$ is a graph and $\lambda:\mathbb{S} = \sqcup S^1\immerses \Gamma$ is an immersion of a disjoint union of cycles.
\end{definition}

We will also restrict our maps to combinatorial maps. 

\begin{definition}
A \emph{combinatorial map} of combinatorial $2$-complexes $f:Y\to X$ is a map that restricts to a combinatorial map of graphs $f_{\Gamma}:\Gamma_Y\to \Gamma_X$ and induces a combinatorial map $f_{\mathbb{S}}:\mathbb{S}_Y\to \mathbb{S}_X$ such that $f_{\Gamma}\circ \lambda_Y = \lambda_X\circ f_{\mathbb{S}}$. We say that $f$ is an \emph{immersion} if $f_{\Gamma}$ is an immersion and $f_{\mathbb{S}}$ restricts to a homeomorphism on each component.
\end{definition}

Since we will always be assuming that our maps are combinatorial, we will often simply neglect to use the descriptor.

The main class of $2$-complexes that we will be working with are \emph{one-relator complexes}. That is, combinatorial $2$-complexes of the form $X = (\Gamma, \lambda)$ where $\lambda:S^1\immerses \Gamma$ is an immersion of a single cycle. Denote by $X_{\lambda}\subset X$ the smallest subcomplex that is a one-relator complex. The following result is the classic Freiheitssatz of Magnus \cite{magnus_30}.

\begin{theorem}
\label{freiheitssatz}
Let $X = (\Gamma, \lambda)$ be a one-relator complex. If $\Lambda\subset \Gamma$ is a connected subgraph in which $\lambda$ is not supported, then $\pi_1(\Lambda)\to \pi_1(X)$ is injective.
\end{theorem}

We will call subgraphs of one-relator complexes that satisfy the hypothesis of Theorem \ref{freiheitssatz}, \emph{Magnus subgraphs}. This is in analogy with \emph{Magnus subgroups}: if $G$ is a one-relator group with one-relator presentation $\langle \Sigma \mid r\rangle$, with $r$ cyclically reduced, then a Magnus subgroup for this presentation is a subgroup generated by a subset $\Lambda\subset \Sigma$ such that $r\notin \langle\Lambda\rangle<F(\Sigma)$.

If $X$ is a one-relator complex and $A\subset X$ is a Magnus subgraph, then $\pi_1(A)$ is a Magnus subgroup for some one-relator presentation of $\pi_1(X)$. We can see this by taking a spanning tree $T\subset \Gamma$ such that $T\cap A$ is a spanning tree for $A$. Then by contracting $T$, we obtain a presentation complex for a one-relator presentation of $\pi_1(X)$ in which $\pi_1(A)$ is a Magnus subgroup.

If $B\subset X$ is another Magnus subgraphs and $A\cap B$ is connected, then, as above, we may obtain a one-relator presentation for $\pi_1(X)$ in which both $\pi_1(A)$ and $\pi_1(B)$ are Magnus subgroups. If $A\cap B$ is not connected, then this is no longer the case. Nevertheless, by adding edges to $A, B$ so that $A\cap B$ is connected, we see that $\pi_1(A)*F$ and $\pi_1(B)*F$ are Magnus subgroups for some one-relator presentation of $\pi_1(X)*F$, where $F$ is a finitely generated free group.

The interactions between Magnus subgroups of one-relator groups are well understood. The following theorems are the main results in \cite{collins_04} and \cite{collins_08} respectively. 

\begin{theorem}
\label{exceptional_intersection}
Let $X = (\Gamma, \lambda)$ be a one-relator complex and let $A, B\subset X$ be Magnus subgraphs with $A\cap B$ connected. Then one of the following holds:
\begin{enumerate}
\item $\pi_1(A)\cap \pi_1(B) = \pi_1(A\cap B)$,
\item $\pi_1(A)\cap \pi_1(B) = \pi_1(A\cap B)*\langle c\rangle$ for some $c\in \pi_1(X)$.
\end{enumerate}
If $\lambda$ is imprimitive, then $\pi_1(A)\cap \pi_1(B) = \pi_1(A\cap B)$.
\end{theorem}

We say a pair of Magnus subgroups $\pi_1(A)$, $\pi_1(B)<\pi_1(X)$ have \emph{exceptional intersection} if the latter situation occurs.

\begin{theorem}
\label{conjugates_intersection}
Let $X = (\Gamma, \lambda)$ be a one-relator complex and let $A, B\subset X$ be Magnus subgraphs with $A\cap B$ connected. Then for any $g\in \pi_1(X)$, one of the following holds:
\begin{enumerate}
\item $\pi_1(A)\cap \pi_1(B)^g = 1$,
\item $\pi_1(A)\cap \pi_1(B)^g = \langle c\rangle$ for some $c\in \pi_1(X)$,
\item $g\in \pi_1(B)\cdot \pi_1(A)$.
\end{enumerate}
If $\lambda$ is imprimitive, then either $\pi_1(A)\cap \pi_1(B)^g = 1$ or $g\in \pi_1(B)\cdot \pi_1(A)$.
\end{theorem}

\begin{remark}
\label{torsion_remark}
Recall that the fundamental group of a one-relator complex $X = (\Gamma, \lambda)$ has torsion if and only if $\lambda$ is imprimitive by \cite{karrass_60}.
\end{remark}

We may, in some sense, strengthen Theorems \ref{exceptional_intersection} and \ref{conjugates_intersection} to incorporate intersections of subgroups of fundamental groups of Magnus subgraphs. First, we will need the following lemma.

\begin{lemma}
\label{Magnus_inert}
Let $X = (\Gamma, \lambda)$ be a one-relator complex and let $A, B\subset X$ be Magnus subgraphs with $A\cap B$ connected. If $\gamma:\Gamma\immerses A$ is a graph immersion such that $\gamma_*\pi_1(\Gamma) = \pi_1(A)\cap \pi_1(B)$, then $\gamma$ is a strongly inert graph immersion. In particular, $\pi_1(A)\cap \pi_1(B)$ is strongly inert in $\pi_1(A)$.
\end{lemma}

\begin{proof}
By Theorem \ref{exceptional_intersection}, $\pi_1(A)\cap \pi_1(B)$ is an echelon subgroup of $\pi_1(A)$ (see \cite[Definitions 3.1 \& 3.3]{rosenmann_13} for the definition of an echelon subgroup of a free group). If $F, G$ are free groups and $J< F$, $K<G$ are echelon subgroups, then by definition, $J*K<F*G$ is an echelon subgroup. Thus, Lemma \ref{strongly_inert}, combined with \cite[Theorem 3.6]{rosenmann_13} implies that $\gamma$ is a strongly inert graph immersion.
\end{proof}

\begin{theorem}
\label{subgroup_Magnus_intersection}
Let $X = (\Gamma, \lambda)$ be a one-relator complex and $A, B\subset X$ be Magnus subgraphs with $A\cap B$ connected. If $C\leqslant\pi_1(A)$ and $D\leqslant \pi_1(B)$ are finitely generated subgroups, with $C$ a strongly inert subgroup of $\pi_1(A)$, then the following is satisfied:
\[
\sum_{\substack{DgC\\ g\in \pi_1(X)}}\rr(C\cap D^g) = \sum_{\substack{DgC\\g\in \pi_1(B)\pi_1(A)}}\rr(C\cap D^g)\leq \rr(D)
\]
\end{theorem}

\begin{proof}
The first equality follows from Theorem \ref{conjugates_intersection}. Since $C\leqslant \pi_1(A)$ and $D\leqslant \pi_1(B)$, we have that
\begin{align*}
\sum_{\substack{DgC\\g\in \pi_1(B)\pi_1(A)}}\rr(C\cap D^g) &\leqslant \sum_{\substack{DbaC\\b\in \pi_1(B), a\in\pi_1(A)}}\rr\left(C^{a^{-1}}\cap D^b\right) \\
					&= \sum_{\substack{DbaC\\b\in \pi_1(B), a\in\pi_1(A)}}\rr\left(C^{a^{-1}}\cap A\cap B\cap D^b\right) \\
					&= \sum_{\substack{DbaC\\b\in \pi_1(B), a\in\pi_1(A)}}\rr\left(C\cap \left((A\cap B)\cap D^b\right)^a\right).
\end{align*}
Next, let $S$ be a set of double coset representatives for $D\backslash \pi_1(B)\pi_1(A)/ C$. For each $g\in S$, choose an element $b\in \pi_1(B)$ such that $b^{-1}g\in\pi_1(A)$. Denote by $S_B\subset\pi_1(B)$ the set of such elements. Denote by $S_b$ the set of distinct $a\in \pi_1(A)$ such that $ba\in S$. Each element in $S_b$ is in a distinct $((\pi_1(A)\cap \pi_1(B))\cap D^b)\backslash \pi_1(A)/C$ double coset. Each element in $S_B$ is in a distinct $D\backslash \pi_1(B)/ (\pi_1(A)\cap\pi_1(B))$ double coset. Then
\begin{align*}
\sum_{\substack{DgC\\g\in \pi_1(B)\pi_1(A)}}\rr\left(C\cap D^g\right) &= \sum_{b\in S_B}\sum_{a\in S_b}\rr\left(C\cap \left((\pi_1(A)\cap \pi_1(B))\cap D^b\right)^a\right)\\
					&\leqslant \sum_{b\in S_B}\rr\left((\pi_1(A)\cap \pi_1(B))\cap D^b\right)\\
					&\leqslant \rr(D)
\end{align*}
where the first inequality follows from the fact that $C$ is strongly inert in $\pi_1(A)$ and the second inequality follows from the fact that $\pi_1(A) \cap \pi_1(B)$ is strongly inert in $\pi_1(B)$ by Lemma \ref{Magnus_inert}.
\end{proof}

%%%%%%%%%%%%%%
\section{Graphs of spaces}
\label{sec:prelim_2}
%%%%%%%%%%%%%%

Let $\Gamma$ be a connected graph. Let $\{X_v\}_{v\in V(\Gamma)}$ and $\{X_e\}_{e\in E(\Gamma)}$ be collections of connected CW-complexes. We call these the \emph{vertex spaces} and \emph{edge spaces} respectively. Let $e\in E(\Gamma)$ and $o(e) = v_-$ and $t(e) = v_+$; then let $\partial^{\pm}_e:X_e\to X_{v_{\pm}}$ be $\pi_1$-injective combinatorial maps. This data determines a \textit{graph of spaces} $\mathcal{X} = (\Gamma, \{ X_v\}, \{X_e\}, \{\partial_e^{\pm}\})$. The \emph{geometric realisation} of $\mathcal{X}$ is defined as follows:
\[
X_{\mathcal{X}} = \left(\bigsqcup_{v\in V(\Gamma)}X_v\sqcup\bigsqcup_{e\in E(\Gamma)}(X_e\times [-1, 1])\right)/\sim
\]
with $(x, \pm 1)\sim \partial^{\pm}_e(x)$ for each $e\in E(\Gamma)$. This space has a CW-complex structure in the obvious way. We will say a cell $c\subset X_{\mathcal{X}}$ is \emph{horizontal} if its attaching map is supported in a vertex space, \emph{vertical} otherwise. 

There is a natural \emph{vertical map} 
\[
\mathsf{v}:X_{\mathcal{X}}\to \Gamma
\]
where $X_v$ maps to $v$ and $X_e\times (-1,1)$ maps to the open edge $e$ in the obvious way. The following fact about the vertical map $\mathsf{v}$ is well known, see \cite{serre_80}.

\begin{lemma}
\label{vertical_map}
The map:
\[
\mathsf{v}_*:\pi_1(X_{\mathcal{X}})\to \pi_1(\Gamma)
\]
is surjective.
\end{lemma}

We may also define a \emph{horizontal map}:
\[
\mathsf{h}:X_{\mathcal{X}}\to H_{\mathcal{X}}
\]
as the quotient map given by the transitive closure of the relation which, for each $e\in E(\Gamma)$, $i\in [-1,1]$ and $x\in X_e$, identifies $(x, i)$ with $\partial^+_e(x)$ and $\partial^-(x)$. A path $p:I\to X_{\mathcal{X}}$ is \emph{vertical} if $\mathsf{h}\circ p$ is a constant path and \emph{horizontal} if $\mathsf{v}\circ p$ is a constant path.

In general, not much can be said about the horizontal map. However, with sufficient restrictions on the edge maps, we can show that it is a homotopy equivalence.

\begin{proposition}
\label{homotopy_equivalence}
Let $\mathcal{X} = (\Gamma, \{X_v\}, \{X_e\}, \{\partial_e^{\pm}\})$ be a graph of spaces. Suppose that $\partial_e^{\pm}$ are given by inclusions of subcomplexes. Then the following hold:
\begin{enumerate}
\item $\mathsf{h}:X_{\mathcal{X}}\to H_{\mathcal{X}}$ is a homotopy equivalence if and only if $\mathsf{h}^{-1}(\mathsf{h}(c))$ is a tree for each $0$-cell $c\in X_{\mathcal{X}}^{(0)}$.
\item If $\mathsf{h}$ is a homotopy equivalence, then $H_{\mathcal{X}}$ has a CW-structure inherited from $X_{\mathcal{X}}$ and $\mathsf{h}\mid X_v$ is an immersion for all $v\in V(\Gamma)$.
\end{enumerate}
\end{proposition}

\begin{proof}
Let $c\subset X_{\mathcal{X}}$ be an open horizontal cell and denote by $\Gamma_c = \mathsf{h}^{-1}(\mathsf{h}(c))$. If $\Gamma_c$ is not a tree for some $0$-cell $c$, then there is a loop in $\Gamma_c\subset X_{\mathcal{X}}$ that maps to a non-trivial loop in $\Gamma$ under the vertical map, but that maps to the trivial loop under $\mathsf{h}$. Thus, $\mathsf{h}$ cannot be a homotopy equivalence. So now let us suppose that $\Gamma_c$ is a tree for all $0$-cells $c\in X_{\mathcal{X}}^{(0)}$.

We show by induction on $n$ that for each horizontal $n$-cell $c\subset X_{\mathcal{X}}$, the preimage $\mathsf{h}^{-1}(\mathsf{h}(c))$ is homeomorphic to $c\times \Gamma_c$ where $\Gamma_c\subset X_{\mathcal{X}}$ is a tree containing only vertical edges and hence immersing into $\Gamma$ via the vertical map. The claim holds for $n = 0$ by the previous paragraph so now assume that $n\geqslant 1$ and the inductive hypothesis holds. If $d\subset X_{\mathcal{X}}$ is an open cell of dimension strictly less than $n$ and such that the image of the attaching map for $c$ intersects $d$, then the immersion $\Gamma_c\immerses\Gamma$ factors through the immersion $\Gamma_d\immerses\Gamma$. Indeed, if $c\subset X_v$ lies in the image of $\partial^{\pm}_e$, then so must $d\subset X_v$. Thus, the edge in $\Gamma_c$ corresponding to $c\times[-1,1]\subset X_e\times [-1,1]$ maps to the edge in $\Gamma_d$ corresponding to $d\times[-1,1]\subset X_e\times[-1,1]$. Here we are using the fact that $\partial^{\pm}_e$ is an inclusion of subcomplexes to naturally identify $c, d\subset X_v$ with the corresponding cells in $X_e$. Since $\Gamma_d$ is a tree by assumption, $\Gamma_c$ must also be a tree and, in particular, $\Gamma_c$ can be seen as a subtree of $\Gamma_d$.

We may now define the cellular structure on $H_{\mathcal{X}}$. Define an equivalence relation on the horizontal cells of $X_{\mathcal{X}}$ defined by $c_1\sim c_2$ if $c_1\subset \mathsf{h}^{-1}(\mathsf{h}(c_2))$ and $c_2\subset \mathsf{h}^{-1}(\mathsf{h}(c_1))$. By the previous paragraph, each open cell $c\in [c]$ maps homeomorphically to its image via $\mathsf{h}$ and each cell in a given equivalence class has the same image. Choosing a representative $c$ for each equivalence class $[c]$, there is an open $n$-cell $\mathsf{h}(c)\subset H_{\mathcal{X}}$ with attaching map given by composing the attaching map for $c$ with the horizontal map $\mathsf{h}$. Since each preimage of an open $n$-cell in $H_{\mathcal{X}}$ is contractible, a standard result now implies that $\mathsf{h}$ is a homotopy equivalence (for instance, see the proof of \cite[Proposition 0.17]{hatcher_00}).

With this cellular structure on $H_{\mathcal{X}}$, we now show that $\mathsf{h}\mid X_v$ is an immersion for each $v\in V(\Gamma)$. It is clearly an immersion on $0$-skeleta. If it is an immersion on $n$-skeleta, but not on $(n+1)$-skeleta, let $x\in X_v$ be a point at which $\mathsf{h}\mid X_v$ is not an immersion on the $(n+1)$-skeleton. Since $\mathsf{h}$ restricts to homeomorphisms on each open cell in $X_v$, we must have that $x$ lies in some open cell $c$ of dimension $n$ or less. Since $\mathsf{h}\mid X_v$ is an immersion on the $n$-skeleton, there are two distinct open $(n+1)$-cells $c_1, c_2\subset X_v$ that lie in the same equivalence class and such that the image of their attaching maps intersect $c$. By definition of the tree $\Gamma_c$, this would then yield a non-trivial loop in $\Gamma_c$ which is a contradiction. Hence $\mathsf\mid X_v$ is an immersion for all $v\in V(\Gamma)$.
\end{proof}

If $\mathcal{X}$ is a graph of spaces, then the universal cover $\tilde{X}_{\mathcal{X}}\to X_{\mathcal{X}}$ also has a graph of spaces structure where each vertex space is the universal cover of some vertex space of $\mathcal{X}$ and each edge space is the universal cover of some edge space of $\mathcal{X}$. We will denote this by $\tilde{\mathcal{X}} = (T, \{\tilde{X}_v\}, \{\tilde{X}_e\}, \{\tilde{\partial}^{\pm}_e\})$ where $T$ is the \emph{Bass--Serre tree of $\mathcal{X}$}. There is a natural covering action of $\pi_1(X_{\mathcal{X}})$ on $\tilde{X}_{\mathcal{X}}$ which pushes forward to an action on $T$. Indeed, we have the following $\pi_1(X_{\mathcal{X}})$-equivariant commuting diagram:
\[
\begin{tikzcd}
\tilde{X}_{\mathcal{X}} = X_{\tilde{\mathcal{X}}} \arrow[r] \arrow[d] & T \arrow[d] \\
X_{\mathcal{X}} \arrow[r]                                             & \Gamma     
\end{tikzcd}
\]

If $\tilde{X}$ satisfies the hypothesis of Proposition \ref{homotopy_equivalence}, then we also have a $\pi_1(X_{\mathcal{X}})$-equivariant commuting diagram:
\[
\begin{tikzcd}
\tilde{H}_{\mathcal{X}} = H_{\tilde{\mathcal{X}}} \arrow[d] & \tilde{X}_{\mathcal{X}} \arrow[r] \arrow[d] \arrow[l] & T \arrow[d] \\
H_{\mathcal{X}}                                             & X_{\mathcal{X}} \arrow[r] \arrow[l]                   & \Gamma     
\end{tikzcd}
\]

\subsection{Hyperbolic graphs of spaces}

Acylindrical actions were first defined by Sela in \cite{sela_97}.

\begin{definition}
Let $G$ be a group acting on a tree $T$. The action is \emph{$k$-acylindrical} if every subset of $T$ of diameter at least $k$ is pointwise stabilised by at most finitely many elements. We will say the action is \emph{acylindrical} if there exists some constant $k>0$ such that the action is $k$-acylindrical.
\end{definition}

By putting constraints on the geometry of the vertex and edge spaces of a graph of spaces $\mathcal{X}$, we may deduce acylindricity of the action of $\pi_1(X_{\mathcal{X}})$ on its Bass--Serre tree. The notion of hyperbolicity that we use is that of Gromov's $\delta$-hyperbolicity. One direction of the following theorem is due to Bestvina--Feighn \cite{bestvina_92_combination} (for the hyperbolicity statement) and Kapovich \cite[Theorem 1.2]{kapovich_01} (for the quasi-convexity statement), while the other direction follows from the fact that quasi-convex subgroups of hyperbolic groups have finite height, due to Gitik--Mitra--Rips--Sageev \cite{gitik_98}.

\begin{theorem}
\label{kapo_main}
Let $\mathcal{X} = (\Gamma, \{X_v\}, \{X_e\}, \{\partial_e^{\pm}\})$ be a graph of spaces such that $X_{\mathcal{X}}$ is compact, $\tilde{X}_v$ is hyperbolic for all $v\in V(\Gamma)$ and $\tilde{\partial}_e^{\pm}$ is a quasi-isometric embedding for all $e\in E(\Gamma)$. Then $\pi_1(X_{\mathcal{X}})$ acts acylindrically on the Bass-Serre tree $T$, if and only if $\tilde{X}_{\mathcal{X}}^{(1)}$ is hyperbolic with the path metric and $\tilde{X}^{(1)}_v\to \tilde{X}^{(1)}_{\mathcal{X}}$ (or $\tilde{X}^{(1)}_e\to \tilde{X}^{(1)}_{\mathcal{X}}$) is a quasi-isometric embedding for all $v\in E(\Gamma)$ (or for all $e\in E(\Gamma)$).
\end{theorem}

%%%%%%%%%%%%%%
\section{$\Z$-domains and one-relator hierarchies}
\label{sec:zdomain}
%%%%%%%%%%%%%%

Let $X$ be a CW-complex. A $\Z$-cover is a connected covering map $p\colon Y\to X$ with $\deck(p)\isom \Z$. We now define $\Z$-domains. These are subcomplexes of $\Z$-covers of $X$ that allow us to construct a homotopy equivalence between $X$ and a graph of spaces in certain situations. We will prove that $\Z$-domains always exist for finite CW-complexes and use this, in combination with the Freiheitssatz, to establish our hierarchy result for one-relator groups.

\begin{definition}
Let $p\colon Y\to X$ be a $\Z$-cover of CW-complexes and denote by $t\in \Z$ a generator. A \emph{$\Z$-domain} for $p$ is a subcomplex $D\subset Y$ with the following properties:
\begin{enumerate}
\item\label{cdt1} $\Z\cdot D = Y$.
\item\label{cdt2} $D\cap t^i\cdot D\subset D\cap t\cdot D\cap \ldots\cap t^i\cdot D$ for all $i>0$.
\item\label{cdt3} $D\cap t\cdot D$ is connected and non-empty.
\end{enumerate}
We will denote by $\zd(p)$ the set of all $\Z$-domains. A \emph{minimal $\Z$-domain} is a $\Z$-domain, minimal under the partial order of inclusion.
\end{definition}

\begin{example}
Let $S_{g}$ be the orientable surface of genus $g\geq 1$ and $\gamma:S^1\to S_g$ a non-separating simple closed curve. We may give $S_g$ a CW structure so that $\gamma$ is in the one-skeleton. The curve $\gamma$ determines an epimorphism $\pi_1(S_g)\to \Z$ via the intersection form. Let $p:Y\to S_g$ be the induced cyclic cover. Consider the subspace $Z\subset Y$ obtained by taking the closure of some component of $p^{-1}(S_g - \Ima(\gamma))$. Its translates cover $Y$ and intersect in lifts of $\gamma$ and so $Z$ is a (minimal) $\Z$-domain for $p$.
\end{example}

Now let $p:Y\to X$ be a $\Z$-cover, $t\in \Z\isom \deck(p)$ a generator and $D\subset Y$ a $\Z$-domain. Consider the following families of spaces:
\[
\{D_i\}_{i\in \Z}
\]
\[
\{E_{i}\}_{i\in\Z}
\]
where $D_i\isom t^i\cdot D$ and $E_{i}\isom t^i\cdot(D\cap t\cdot D)$. There are natural inclusion maps:
\[
\iota_{i}^{-}\colon E_{i}\to D_i
\]
and
\[
\iota_{i}^+\colon E_i\to D_{i+1}
\]
given by the inclusions $t^i\cdot(D\cap t\cdot D)\to t^i\cdot D$ and $t^i\cdot(D\cap t\cdot D)\to t^{i+1}\cdot D$. With this data we define the space:
\[
\Theta(p, D) = \left(\bigsqcup_{i\in \Z}D_i\right)\sqcup\left(\bigsqcup_{i\in\Z}E_{i}\times[-1, 1]\right)/\sim
\]
where $(x, \pm1)\sim \iota_{i}^{\pm}(x)$ for all $i\in \Z$ and $x\in E_{i}$. We also have an action of $\Z$ on $\Theta(p, D)$ induced by the action on $\{D_i\}_{i\in \Z}$.

\begin{proposition}
\label{splitting}
Let $p\colon Y\to X$ be a $\Z$-cover with $t\in \deck(p)$ a generator and let $D\subset Y$ be a $\Z$-domain. If $D\cap t\cdot D\to D, t\cdot D$ are $\pi_1$-injective, then $\Theta(p, D)$ and $\Z\backslash \Theta(p, D)$ are graphs of spaces and the induced maps
\begin{align*}
\Theta(p, D)&\to Y\\
\Z\backslash\Theta(p, D)&\to X
\end{align*}
are homotopy equivalences factoring as horizontal maps composed with homeomorphisms. In particular, $\pi_1(X)$ splits as a HNN-extension with vertex group $\pi_1(D)$ and edge group $\pi_1(D\cap t\cdot D)$.
\end{proposition}

\begin{proof}
Condition (\ref{cdt3}) of $\Z$-domains combined with $\pi_1$-injectivity of $D\cap t\cdot D\to D, t\cdot D$ implies that $\Theta(p, D)$ and $\Z\backslash \Theta(p, D)$ are graphs of spaces. Condition (\ref{cdt1}) implies that the maps $\Theta(p, D)\to Y$ and $\Z\backslash \Theta(p, D)\to X$ are surjective and the definition of the horizontal map $\mathsf{h}$ shows that they factor through the horizontal maps. The preimages of points in $Y$ under the map $\Theta(p, D)\to Y$ are collections of intervals which can be identified with subintervals of $\R$, the underlying graph of the graph of spaces $\mathcal{Y}$ for $\Theta(p, D) = X_{\mathcal{Y}}$. The horizontal map $\Theta(p, D) = X_{\mathcal{Y}}\to H_{\mathcal{Y}}$ is the quotient map defined by identifying each connected component of these preimages to points. Condition (\ref{cdt2}) of $\Z$-domains implies these preimages are actually connected. Hence, the factored map $H_{\mathcal{Y}}\to Y$ is actually a homeomorphism. Since $\Z$ acts freely on $Y$, distinct points in a single point preimage in $\Theta(p, D)$ lie in distinct $\Z$-orbits and so it follows from the commutativity of the diagram
\[
\begin{tikzcd}
{\Theta(p, D)} \arrow[r] \arrow[d] & {\mathbb{Z}\backslash \Theta(p, D)} \arrow[d] \\
Y \arrow[r]                        & X                                            
\end{tikzcd}
\]
that preimages of points in $X$ under the map $\Z\backslash\Theta(p, D)\to X$ are also connected subintervals of $\R$. As before, we obtain that the map $\Z\backslash \Theta(p, D) \to X$ factors as the horizontal map composed with a homeomorphism. By Proposition \ref{homotopy_equivalence}, both horizontal maps are homotopy equivalences and so the induced maps are also homotopy equivalences.
\end{proof}

\subsection{Existence of $\Z$-domains}
\label{existence_section}

Let us first fix some notation. Let $p\colon Y\to X$ be a $\Z$-cover with $t$ a generator for $\deck(p)$. Choose some spanning tree $T\subset X^{(1)}$ and some orientation on $X^{(1)}$. This induces an identification of $\pi_1\left(X^{(1)}\right)$ with $F(\Sigma)$, the free group generated by
\[
\Sigma = \{e\}_{e\in E(X^{(1)})\setminus E(T)}.
\]
For any subset $A\subset \Sigma$, define:
\[
T^A = T\cup\left(\bigcup_{e\in A}e\right).
\]
Choose some lift of $T$ to $Y$ and denote this by $T_0\subset Y$. Then, since $p$ is regular, every lift of $T$ is obtained by translating $T_{0}$ by an element of $\Z$. So we will denote by:
\[
T_{i} = t^i\cdot T_{0}.
\]
For each $e\in E(X^{(1)})$, denote by $e_i$ the lift of $e$ such that $o(e_i)\in T_i$. If $e\notin \Sigma$, then $t(e_i)\in T_i$. If $e\in \Sigma$, then $t(e_i)\in T_{i + \iota(e)}$ where 
\[
\iota:\Sigma\to \pi_1(X^{(1)}, x)\to \Z
\]
is the obvious map.

If $e\in \Sigma$ and $C\subset \Z$, define $C_e\subset C$ to be the subset consisting of the elements $i\in C$ such that $i + \iota(e)\in C$. Then if $A\subset \Sigma$, we define the following subcomplex of the 1-skeleton of $Y$:
\[
T_C^A = \left(\bigcup_{i\in C}T_i\right)\cup\left(\bigcup_{e\in A}\bigcup_{i\in C_e}e_i\right)
\]
Examples of such subcomplexes can be seen in Figures \ref{cyclic_tree} and \ref{minimal_cyclic_tree}. We also write
\begin{align*}
K_A &= \langle \iota(A)\rangle < \Z\\
k_A &= [\Z:K_A].
\end{align*}

\begin{remark}
\label{connected_remark}
The subgroup $K_A$ acts on each component of $T^{A}_{\Z}$ and the quantity $k_A$ is precisely the number of connected components of $T_{\Z}^A$. In fact, if we contract all the lifts of $T$ in $Y^{(1)}$, then the resulting graph is the Cayley graph for $\Z$ over the generating set $\iota(\Sigma)$. So the connected components of $T_{\Z}^A$ correspond to $K_A$ cosets in $\Z$.
\end{remark}

We will call a subset $C\subset \Z$ \emph{connected} if for all triples of integers $i\leqslant j\leqslant k$, if $i, k\in C$ then $j\in C$.

\begin{proposition}
\label{connectedness}
Let $p:Y\to X$ be a $\Z$-cover. For any subset $A\subset \Sigma$ and any connected subset $C\subset\mathbb{Z}$, the following hold:
\begin{enumerate}
\item\label{itm:kA_infty} If $k_A = \infty$, then $T_C^A$ consists of $|C|$ connected components.
\item\label{itm:kA_finite} If $k_A<\infty$, then there exists some constant $k = k(A)\geq 0$ such that $T_C^A$ consists of $k_A$ connected components whenever $|C|\geq k$.
\end{enumerate}
\end{proposition}

\begin{proof}
For (\ref{itm:kA_infty}), note that we have that $K_A = \{0\}$, hence each edge $e_j$ with $e\in A$ has both endpoints in $T_j$. So for each $j\in C$, the subcomplexes $T_{j}^A\subset T_C^A$ are all pairwise disjoint and cover $T_C^A$.

Now assume that $k_A<\infty$. Let $A'\subset A$ be any finite subset such that $\gcd(\iota(A')) = \gcd(\iota(A)) = k_A$. For any pair of connected subsets $C\subset D\subset \Z$ of size larger than the largest element of $\iota(A')$ in absolute value, the inclusion $T_C^{A'}\injects T_D^{A'}$ is surjective on components. Hence, for large enough $C$, the number of components must stabilise at $k_A$ by Remark \ref{connected_remark}. By definition of $T_C^A$, we also have that $T_C^A$ has $k_A$ connected components for all connected $C\subset \Z$ such that $|C|\geqslant k = k(A')$. This establishes (\ref{itm:kA_finite}).
\end{proof}

\begin{proposition}
\label{tree_existence}
Every $\Z$-cover of a finite CW-complex has a finite $\Z$-domain.
\end{proposition}

\begin{proof}
Let $X$ be a finite CW-complex and $p:Y\to X$ a $\Z$-cover. We prove this by induction on $n$-skeleta of $X$. Denote by $p_n:Y^{(n)}\to X^{(n)}$ the restriction of $p$ to the $n$-skeleton of $Y$. Since $X$ is finite, $|\iota(\Sigma)|$ is finite. As $k_{\Sigma} = 1$, we may apply Proposition \ref{connectedness} and obtain an integer $k(\Sigma)\geq 0$ such that $T_C^{\Sigma}$ is connected for all $|C|\geq k$. So if we choose $C$ so that $|C|$ is greater than $k$ and greater than the maximal size of an element in $\iota(\Sigma)$ in absolute value, then $T_C^{\Sigma}$ satisfies conditions (\ref{cdt1}) and (\ref{cdt2}) of the definition of $\Z$-domains. By possibly enlarging $C$ by one element, we may ensure that (\ref{cdt3}) is also satisfied and hence that $T_C^{\Sigma}$ is a $\Z$-domain for $p_1:Y^{(1)}\to X^{(1)}$.

Now suppose we have a $\Z$-domain $D_{n-1}$ for $p_{n-1}:Y^{(n-1)}\to X^{(n-1)}$. Then for any connected subset $C\subset \Z$, the subcomplex $\bigcup_{j\in C}t^j\cdot D_{n-1}$ is also a $\Z$-domain for $p_{n-1}$. The first and last properties are clear, whereas the second property holds because 
\[
\left(\bigcup_{j\in C}t^j\cdot D_{n-1}\right)\cap t^i\cdot \left(\bigcup_{j\in C}t^j\cdot D_{n-1}\right) = \bigcup_{m + i\leqslant j\leqslant M}t^j\cdot D_{n-1}
\]
if $i\leqslant |C|$, where $m = \min{\{C\}}$ and $M = \min{\{C\}}$, and 
\begin{align*}
\left(\bigcup_{j\in C}t^j\cdot D_{n-1}\right)\cap t^i\cdot \left(\bigcup_{j\in C}t^j\cdot D_{n-1}\right) &= t^{M}\cdot D_{n-1}\cap t^i\cdot D_{n-1}\\
				&\subseteq t^{M}\cdot D_{n-1}\cap D_{n-1}^{M+1}\cap \ldots\cap t^i\cdot D_{n-1}
\end{align*}
otherwise, where the last inclusion uses the fact that $D_{n-1}$ was a $\Z$-domain for $p_{n-1}$. We may choose $C$ large enough so that $\bigcup_{j\in C}t^j\cdot D_{n-1}$ contains a lift of each attaching map of $n$-cells in $X$. Then the full subcomplex of $Y^{(n)}$ containing $\bigcup_{j\in C}t^j\cdot D_{n-1}$ is a $\Z$-domain for $p_n$.
\end{proof}

\subsection{Graphs and one-relator complexes}

If we restrict our attention to graphs, the topology of minimal $\Z$-domains is much simpler.

\begin{lemma}
\label{minimal_char}
Let $\Gamma$ be a finite graph and let $p:Y\to \Gamma$ be a $\Z$-cover. If $D\in \zd(p)$ is a minimal $\Z$-domain, then
\begin{align*}
\chi(D) &= \chi(\Gamma) + 1,\\
\chi(D\cap t\cdot D) &= 1.
\end{align*}
\end{lemma}

\begin{proof}
Up to contracting a spanning tree, we may assume that $\Gamma$ has a single vertex and $n$ edges. Let $D\in \zd(p)$ be a minimal $\Z$-domain. We have $\chi(D\cap t\cdot D) = \chi(D) + n - 1$ since $D\cap t\cdot D$ is obtained from $D$ by removing the lowest vertex and each of the $n$ lowest edges. We claim that $D\cap t\cdot D$ is a tree. If not, then there would be an embedded cycle $S^1\injects D\cap t\cdot D$. By possibly replacing this cycle with a translate, we may assume that it traverses an edge $e\subset D$ such that $t\cdot e$ does not lie in $D$. Then removing this edge will leave us with a smaller $\Z$-domain, contradicting minimality. Thus, we must have that $D\cap t\cdot D$ is actually a tree, otherwise we could remove some edge. Thus $\chi(D) = \chi(\Gamma) + 1$ and $\chi(D\cap t\cdot D) = 1$.
\end{proof}

As a consequence of Lemma \ref{minimal_char} and Proposition \ref{splitting}, we can see that minimal $\Z$-domains of $\Z$-covers of graphs $\Gamma$ correspond to certain free product decompositions of $\pi_1(\Gamma)$.

\begin{example}
\label{primitive_example}
Let $\Gamma$ be a graph with a single vertex and two edges, so $\chi(\Gamma) = -1$. The spanning tree is the unique vertex $v\in \Gamma$. Let $\Sigma = \{a, b\}$ be the two edges and let $\phi:F(a, b)\to \Z$ be the homomorphism that sends $a$ to $5$ and $b$ to $3$. Then let $p:Y\to \Gamma$ be the corresponding $\Z$-cover. Let $C= \{0, 1, ..., 6\}$, then $T_C^{\Sigma}$ can be seen in Figure \ref{cyclic_tree}. The numbers $i$ in the diagram correspond to $T_i$, the lower black edges are lifts of $a$ and the upper red edges are lifts of $b$ edges. This is not quite a $\Z$-domain as its intersection with a translate is disconnected. However, we can see that $T_{C\cup 7}^{\Sigma}$ is a genuine $\Z$-domain, see Figure \ref{minimal_cyclic_tree}. In fact, it is the unique minimal $\Z$-domain for $p$. The unique primitive cycle that factors through $T_{C\cup 7}^{\Sigma}$, represents the unique conjugacy class of primitive element in $\ker(\phi)$.
\end{example}

\begin{figure}
\centering
\includegraphics[scale=0.8]{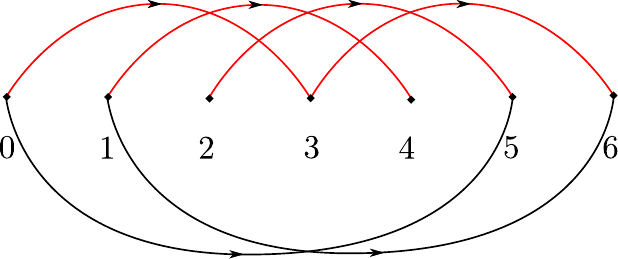}
\caption{Not quite a $\Z$-domain.}
\label{cyclic_tree}
\end{figure}

\begin{figure}
\centering
\includegraphics[scale=0.8]{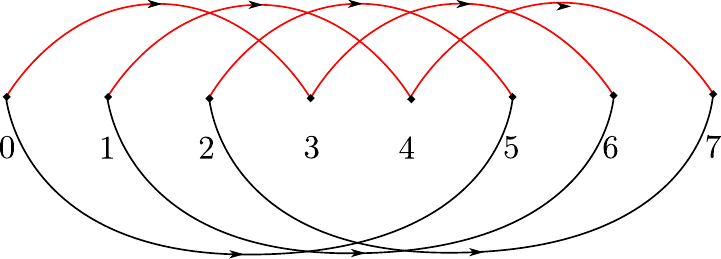}
\caption{A minimal $\Z$-domain.}
\label{minimal_cyclic_tree}
\end{figure}

Another special case is that of one-relator complexes. The following two propositions are the foundations for our one-relator hierarchies.

\begin{proposition}
\label{one-relator_domain}
Let $X$ be a finite one-relator complex, let $p\colon Y\to X$ be a $\Z$-cover and let $Z\subset \zd(p)$ be a minimal $\Z$-domain. Then $Z$ is a finite one-relator complex, such that either $Z\cap t\cdot Z$ is a tree, or $t^{-1}\cdot Z\cap Z\cap t\cdot Z$ is connected.
\end{proposition}

\begin{proof}
Let $X = (\Gamma, \lambda)$ be a finite one-relator complex and $p:Y\to X$ a $\Z$-cover. Let $Z\subset Y$ be a minimal $\Z$-domain for $p$. There is a minimal $\Z$-domain $D\subset Y^{(1)}$ for the $\Z$-cover $Y^{(1)}\to X^{(1)}$ contained in $Z$. By Lemma \ref{minimal_char}, $D\cap t\cdot D$ is a tree. Just as in the proof of Proposition \ref{tree_existence}, we have that $\bigcup_{i\in C}t^i\cdot D$ is also a $\Z$-domain for all connected subsets $C\subset \Z$. Now let $C = \{0, \ldots, m\}$ be of smallest size such that a lift of $\lambda$ is supported in $\bigcup_{i\in C}t^i\cdot D$. Since $D\cap t\cdot D$ is a tree, at most one lift $\tilde{\lambda}$ can be supported in this subgraph. Now the union of this subgraph with the two-cell attached along $\tilde{\lambda}$ is a $\Z$-domain which must contain a translate of $Z$ by minimality. Thus, $Z$ is a one-relator complex. By construction, either $C = \{0\}$ and so $Z\cap t\cdot Z = D\cap t\cdot D$ and so is a tree, or $|C|>1$ and so $Z\cap t\cdot Z$ is a $\Z$-domain for $Y^{(1)}\to X^{(1)}$ and so $t^{-1}\cdot Z\cap Z\cap t\cdot Z$ is connected.
\end{proof}

\begin{remark}
More generally, by \cite[Lemma 2]{howie_87}, if $X$ is a staggered two-complex and $p:Y\to X$ a cyclic cover, then every $\Z$-domain $D\in \zd(p)$ is also a staggered two-complex. Furthermore, by \cite[Corollary 6.2]{hruska_01} and Proposition \ref{splitting}, each $D\in \zd(p)$ induces a HNN-splitting of $\pi_1(X)$ over $\pi_1(D)$.
\end{remark}

\subsection{One-relator hierarchies}
\label{sec:hierarchy}

Define the \emph{complexity} of a one-relator complex $X = (\Gamma, \lambda)$ to be the following quantity:
\[
c(X) = \left(\frac{\abs{\lambda}}{\deg(\lambda)} - \abs{X_{\lambda}^{(0)}}, -\chi(X)\right).
\]
We endow the complexity of one-relator complexes with the dictionary order so that $(q, r)< (s, t)$ if $q<s$ or $q = s$ and $r<t$. Note that the first component of $c(X)$ is a slight modification of the notion of complexity used by Howie in \cite{howie_81}.

\begin{proposition}
\label{complexity}
Let $X$ be a finite one-relator complex and $p:Y\to X$ a cyclic cover. If $D\in \zd(p)$ is a minimal $\Z$-domain, then:
\[
c(D)<c(X).
\]
\end{proposition}

\begin{proof}
It is clear that if $c(D) = (q, r)$ and $c(X) = (s, t)$, we cannot have $q> s$. Suppose that we have $q = s$, then we must have that $X_{\lambda}$ actually lifts to all $\Z$-domains. But then this implies that $X_{\lambda}$ is a subcomplex of $D$. So applying Lemma \ref{minimal_char} to the induced $\Z$-cover of the graph $X/X_{\lambda}$, we see that $\chi(D) = \chi(X) + 1$ when $D$ is minimal. Thus, $r<t$ and $c(D)<c(X)$ when $D$ is minimal.
\end{proof}

By Proposition \ref{one-relator_domain}, if $X$ is a one-relator complex and $p:Y\to X$ is a cyclic cover, there exists a one-relator complex $X_1\in \zd(p)$. Repeating this, we obtain a sequence of immersions of one-relator complexes $X_N\immerses...\immerses X_1\immerses X_0 = X$ where $X_i$ is a one-relator $\Z$-domain of a $\Z$-cover of $X_{i-1}$ for each $i$. We will call this a \emph{one-relator tower}. Note that each immersion is a tower map in the sense of \cite{howie_81}. If $X_N$ does not admit any $\Z$-covers, we will call this a \emph{maximal} one-relator tower. 

\begin{proposition}
\label{tower}
Every finite one-relator complex $X = (\Gamma, \lambda)$ has a maximal one-relator tower $X_N\immerses ...\immerses X_1\immerses X_0 = X$.
\end{proposition}

\begin{proof}
The proof is by induction on $c(X)$, the base case is $c(X) = (0, 0)$. We have $c(X) = (0, 0)$ precisely when $\Gamma\simeq S^1$. This is the only case where $X$ does not admit any $\Z$-cover. This is because if $c(X)>(0, 0)$, then $\chi(X)\leq 0$ and so $\rk(H_1(X, \Z))\geq 1$. Hence, the base case holds. The inductive step is simply Proposition \ref{complexity}.
\end{proof}

Now we are ready to prove our version of the Magnus--Moldavanskii hierarchy for one-relator groups.

\begin{theorem}
\label{new_hierarchy}
Let $X$ be a finite one-relator complex. There exists a finite sequence of immersions of one-relator complexes:
\[
X_N\immerses ...\immerses X_1\immerses X_0 = X
\]
such that $\pi_1(X_i)\isom \pi_1(X_{i+1})*_{\psi_i}$ where $\psi_i$ is induced by an identification of Magnus subgraphs, and such that $\pi_1(X_N)$ is finite cyclic.
\end{theorem}

\begin{proof}
By Proposition \ref{tower} there is a maximal one-relator tower $X_N\immerses ...\immerses X_1\immerses X_0 = X$. By Theorem \ref{freiheitssatz}, the inclusions $X_i\cap t\cdot X_i\injects X_i, t\cdot X_i$ are all injective on $\pi_1$. Thus, by Proposition \ref{splitting}, $\pi_1(X_i)\isom \pi_1(X_{i+1})*_{\psi_i}$ for some isomorphisms $\psi_i$ induced by an identification of Magnus subgraphs.
\end{proof}

We will call a splitting $\pi_1(X_i)\isom \pi_1(X_{i+1})*_{\psi}$ as in Theorem \ref{new_hierarchy} a \emph{one-relator splitting}. Then the \emph{associated Magnus subgraphs inducing $\psi$} are $A = t^{-1}\cdot X_{i+1}\cap X_{i+1}$ and $B = X_{i+1}\cap t\cdot X_{i+1}$. 

\begin{remark}
\label{magnus_remark}
By Proposition \ref{one-relator_domain}, each $\pi_1(X_{i+1})$ from Theorem \ref{new_hierarchy} has a one-relator presentation so that $\psi_i$ actually identifies two Magnus subgroups for this presentation.
\end{remark}

We will call a one-relator tower $X_N\looparrowright...\looparrowright X_1\looparrowright X_0 = X$ a \emph{hierarchy of length $N$} if $\pi_1(X_N)$ splits as a free product of cyclic groups. Denote by $h(X)$ the \textit{hierarchy length} of $X$, that is, the smallest integer $N$ such that a hierarchy for $X$ of length $N$ exists. By Theorem \ref{new_hierarchy} this is well-defined for all one-relator complexes. We extend this definition also to one-relator presentations $\langle \Sigma \mid w\rangle$ by saying that the hierarchy length of a one-relator presentation is the hierarchy length of its presentation complex. A hierarchy $X_N\looparrowright...\looparrowright X_1\looparrowright X_0 = X$  is of \textit{minimal length} if $N = h(X)$.

The hierarchy length of a one-relator complex is not preserved under homotopy equivalence. This is illustrated by the following examples.

\begin{example}
\label{hierarchy_length_example}
Let $\pm 1\neq p, q\in \Z$ be a pair of coprime integers. Denote by $\pr_{p, q}(x, y)$ any cyclically reduced element in the unique conjugacy class of primitive elements in $\ab^{-1}(p, q)$ where $\ab:F(x, y)\to \Z^2$ is the abelianisation map. Then the Baumslag--Solitar group $\bs(p, q)$ has two one-relator presentations:
\[
\bs(p, q) \isom \langle a, t | t^{-1}a^pt = a^q \rangle \isom \langle b, s | \pr_{p, -q}(b, s^{-1}bs) \rangle
\]
such that the hierarchy length of the first presentation is 2, but of the second is 1. Since presentation complexes of one-relator groups without torsion are aspherical by Lyndon's identity theorem\cite{lyndon_50}, the two associated presentation complexes are homotopy equivalent. 

Note that the words $t^{-1}a^pta^{-q}$ and $\pr_{p, -q}(a, t^{-1}at)$ are not equivalent under the action of $\aut(F(a, t))$ if $\abs{p}, \abs{q}\geq 2$. Similarly for $t^{-1}a^pta^{-q}$ and $\pr_{p, -q}(a, t^{-1}at)^{-1}$.

We may do the same with the Baumslag--Gersten groups:
\[
BG(p, q) \isom \langle a, t |t^{-1}ata^pt^{-1}a^{-1}t = a^q \rangle \isom \langle b, s | \pr_{p, -q}(b, s^{-1}bsbs^{-1}b^{-1}s)\rangle
\]
The hierarchy length of the first presentation is three, but of the second is two.
\end{example}

%%%%%%%%%%%%%%%%%%%%%%%%%%%%%
\section{Normal forms and quasi-convex one-relator hierarchies}
\label{sec:normal_forms}
%%%%%%%%%%%%%%%%%%%%%%%%%%%%%

\subsection{Normal forms}

Computable normal forms for one-relator complexes can be derived from the original Magnus hierarchy \cite{magnus_30}. Following the same idea, we build normal forms for universal covers of one-relator complexes. The geometry of these normal forms will then allow us to prove our main theorem.

Let $X$ be a combinatorial $2$-complex. We define $I(X)$ to be the set of all combinatorial immersed paths. If $c\in I(X)$, then we denote by $\bar{c}$ the reverse path. If $a, b\in I(X)$ are paths with $t(a) = o(b)$, we will write $a*b$ for their concatenation.

\begin{definition}
A \emph{normal form} for $X$ is a map $\eta:X^{(0)}\times X^{(0)}\to I(X)$ where $o(\eta(p,q)) = p$ and $t(\eta(p,q)) = q$ and $\eta(p,p)$ is the constant path, for all $p, q\in X^{(0)}$. We say that $\eta$ is \emph{prefix-closed} if for every $p, q\in X^{(0)}$ and $i\in [0,\abs{\eta(p,q)}]$, we have that $\eta(p,r) = \eta(p,q)\mid_{[0, i]}$ where $r = (\eta(p,q))(i)$.
\end{definition}

We will be particularly interested in certain kinds of normal forms, defined below, the first being quasi-geodesic normal forms and the second being normal forms relative to a given subcomplex. The motivation for these definitions is the following: if $X$ is a two-complex, $Z\subset X$ a $\pi_1$-injective subcomplex such that the universal cover $\tilde{X}$ admits quasi-geodesic normal forms relative to a copy of the universal cover $\tilde{Z}\subset \tilde{X}$, then $\pi_1(Z)$ will be undistorted in $\pi_1(X)$.

\begin{definition}
Let $K>0$. A normal form $\eta:X^{(0)}\times X^{(0)}\to I(X)$ is \emph{$K$-quasi-geodesic} if every path in $\Ima(\eta)$ is a $K$-quasi-geodesic. We will call $\eta$ \emph{quasi-geodesic} if it is $K$-quasi-geodesic for some $K>0$.
\end{definition}

\begin{definition}
Suppose $Z\subset X$ is a subcomplex. A normal form $\eta:X^{(0)}\times X^{(0)}\to I(X)$ is a \emph{normal form relative to $Z$} if, for any $p, q$ contained in the same connected component of $Z$ and $r\in X$, the following hold:
\begin{enumerate}
\item $\eta(p,q)$ is supported in $Z$,
\item\label{itm:relative_normal} $\eta(p,r)*\bar{\eta}(q,r)$ is supported in $Z$ after removing backtracking.
\end{enumerate}
\end{definition}

By definition, any normal form for $X$ is a normal form relative to $X^{(0)}$. 

The general strategy for building our normal forms for universal covers of one-relator complexes will be to use one-relator hierarchies and induction. In order to do so, we must first discuss normal forms for graphs of spaces.

\subsection{Graph of spaces normal forms}
\label{gog_normal_forms}

For simplicity, we only discuss normal forms for graphs of spaces with underlying graph consisting of a single vertex and single edge. However, the construction is the same for any graph.

Let $\mathcal{X} = (\Gamma, \{X\}, \{C\}, \partial^{\pm})$ be a graph of spaces where $\Gamma$ consists of a single vertex and a single edge, $X$ and $C$ are combinatorial two-complexes and $\partial^{\pm}$ are inclusions of subcomplexes. Denote by $A = \Ima(\partial^-)$ and $B = \Ima(\partial^+)$. We may orient the vertical edges so that they go from the subcomplex $A$ to the subcomplex $B$. When it is clear from context which path we mean, we will abuse notation and write $t^{k}$ for all $k\geq 0$, to denote a path that follows $k$ vertical edges consecutively, respecting this orientation. We will write $t^{-k}$ for all $k\geq 0$, for the path that follows $k$ vertical edges in the opposite direction. Since $\partial^{\pm}$ are embeddings, these paths are uniquely defined given an initial vertex.

The universal cover $\tilde{X}_{\mathcal{X}}\to X_{\mathcal{X}}$ also has a graph of spaces structure with underlying graph the Bass--Serre tree of the associated splitting, the vertex spaces are copies of the universal cover $\tilde{X}$, the edge spaces are copies of the universal cover $\tilde{C}$ and the edge maps are $\pi_1(X)$-translates of lifts $\tilde{\partial}^{\pm}:\tilde{C}\to \tilde{X}$ of the maps $\partial^{\pm}$. We will denote by $\tilde{\mathcal{X}}$ the underlying graph of spaces so that $X_{\tilde{\mathcal{X}}} = \tilde{X}_{\mathcal{X}}$. Every path $c\in I(X_{\tilde{\mathcal{X}}})$ can be uniquely factorised:
\[
c = c_0*t^{\epsilon_1}*c_1*...*t^{\epsilon_n}*c_n
\]
with $\epsilon_i = \pm 1$ and where each $c_i$ is (a possibly empty path) supported in some copy of $\tilde{X}$. We say $c$ is \emph{reduced} if there are no two subpaths $c_i, c_j$ that are both supported in the same copy of $\tilde{X}$. Hence that no subpath $c_i*t^{\epsilon_{i+1}}*...*t^{\epsilon_j}*c_j$ has both endpoints in the same copy of $\tilde{X}$. In other words, the path $v\circ c$ is immersed, where $v:\tilde{X}_{\mathcal{X}}\to T$ is the vertical map to the Bass--Serre tree $T$. 

A \emph{vertical square} in $\tilde{X}_{\mathcal{X}}$ is a two-cell with boundary path $e*t^{\epsilon}*f*t^{-\epsilon}$ where $e$ and $f$ are edges in two different copies of $\tilde{X}$. We will call the paths $e*t^{\epsilon}$, $t^{\epsilon}*f$, $f*t^{-\epsilon}$, $t^{-\epsilon}*e$ and their inverses \emph{corners} of this square. We may pair up two corners $\alpha$ and $\beta$ if their concatenation $\alpha*\bar{\beta}$ forms the boundary of a square. In this way, each corner exhibits a vertical homotopy to the opposing corner it is paired up with. For instance, there is a vertical homotopy through the square between $t^{\epsilon}*\bar{f}$ and $e*t^{\epsilon}$. 

Now let 
\[
\eta_{\tilde{A}}:\tilde{X}^{(0)}\times\tilde{X}^{(0)}\to I(\tilde{X})
\]
\[
\eta_{\tilde{B}}:\tilde{X}^{(0)}\times\tilde{X}^{(0)}\to I(\tilde{X})
\]
be $\pi_1(X)$-equivariant normal forms relative to $\tilde{A}$ and $\tilde{B}$ respectively. Since these are $\pi_1(X)$-equivariant, it is not important which lift of $\tilde{A}$ and $\tilde{B}$ we choose. Let
\[
\eta_{\tilde{X}}:\tilde{X}^{(0)}\times\tilde{X}^{(0)}\to I(\tilde{X})
\]
be a $\pi_1(X)$-equivariant normal form for $\tilde{X}$. From this data, we may define normal forms for $X_{\tilde{\mathcal{X}}}$ as follows. We say a normal form
\[
\eta:X_{\tilde{\mathcal{X}}}^{(0)}\times X_{\tilde{\mathcal{X}}}^{(0)}\to I(X_{\tilde{\mathcal{X}}})
\]
is \emph{a graph of spaces normal form induced by $(\{\eta_{\tilde{X}}\}, \{\eta_{\tilde{A}}, \eta_{\tilde{B}}\})$} if the following holds for all $c = c_0*t^{\epsilon_1}*c_1*...*t^{\epsilon_n}*c_n\in \Ima(\eta)$:
\begin{enumerate}
\item $c_0 = \eta_{\tilde{X}}(o(c_0), t(c_0))$.
\item If $\epsilon_i = 1$ then $c_i = \eta_{\tilde{B}}(o(c_i), t(c_i))$. Furthermore, if $e$ is the first edge $c_i$ traverses, then $t*e$ does not form a corner of a vertical square.
\item If $\epsilon_i = -1$ then $c_i = \eta_{\tilde{A}}(o(c_i), t(c_i))$. Furthermore, if $e$ is the first edge $c_i$ traverses, then $t^{-1}*e$ does not form a corner of a vertical square.
\end{enumerate}
Our definition of graph of spaces normal form and the following theorem are the topological translations of the algebraic definition and the normal form theorem \cite[p. 181]{lyn_00}.

\begin{theorem}
\label{gos_normal_form}
If $\eta_{\tilde{A}}, \eta_{\tilde{B}}$ and $\eta_{\tilde{X}}$ are as above, there is a unique graph of spaces normal form $\eta$ induced by $(\{\eta_{\tilde{X}}\}, \{\eta_{\tilde{A}}, \eta_{\tilde{B}}\})$. Moreover, the following are satisfied:
\begin{enumerate}
\item $\eta$ is $\pi_1(X_{\mathcal{X}})$-equivariant,
\item the action of $\pi_1(X)$ on $\eta$ and $\eta_{\tilde{X}}$ coincide,
\item if $\eta_{\tilde{X}}$, $\eta_{\tilde{A}}$ and $\eta_{\tilde{B}}$ are prefix-closed, then so is $\eta$.
\end{enumerate}
\end{theorem}

Before moving on, we shall provide a brief sketch of the proof of Theorem \ref{gos_normal_form}.

Given existence and uniqueness of graph of spaces normal forms, the fact that they must be prefix closed provided $\eta_{\tilde{X}}, \eta_{\tilde{A}}$ and $\eta_{\tilde{B}}$ are can be shown directly from the definition of graph of spaces normal forms. Indeed, any prefix of a path satisfying the three conditions, will also satisfy the three conditions. Similarly for the equivariance claims.

The existence of normal forms can be shown by choosing paths $c = c_0*t^{\epsilon_1}*c_1*\ldots *t^{\epsilon_n}*c_n$ between each pair of points and putting them into normal form as follows. If $n = 0$, then $\eta(o(c), t(c)) = \eta_{\tilde{X}}(o(c), t(c))$. If $n\geqslant 1$ and $\epsilon_n = 1$, then replace $c_n$ with $\eta_{\tilde{B}}(o(c_n), t(c_n)) = b_n*c_n'$, where $b_n$ is the largest prefix contained in the same copy of $\tilde{B}$ that $o(c_n)$ was in. Then perform a sequence of vertical homotopies through vertical squares, replacing $t^{\epsilon_n}*b_n$ with $a_n*t^{\epsilon_n}$, where $a_n$ lies in the same copy of $\tilde{A}$ that $o(t^{\epsilon_n})$ was in. Then $\eta(o(c), t(c))$ can be defined inductively on $n$ by setting $\eta(o(c), t(c)) = \eta(o(c'), t(c'))*c_n'$, where $c' = c_0*t^{\epsilon_1}*c_1*\ldots*t^{\epsilon_{n-1}}*c_{n-1}*a_n$, after possibly removing backtracking. The case where $\epsilon_n = -1$ is handled similarly. The fact that these are graph of spaces normal forms is by construction. This procedure will be important for the proof of Theorem \ref{hyperbolic_normal_forms}.

Now suppose that $c = c_0*t^{\epsilon_1}*\ldots *t^{\epsilon_n}*c_n$ and $c' = c_0'*t^{\eta_1}*\ldots *t^{\eta_m}*c'_m$ are two paths in normal form with the same origin and target vertices. Since the underlying graph of the graph of spaces $\tilde{\mathcal{X}}$ is a tree, the Bass--Serre tree of the graph of groups induced by $\mathcal{X}$, it follows that $n = m$ and $\epsilon_i = \eta_i$ for all $i$. Moreover, it also follows that $c_n*\bar{c}_n'$ must be equal if $n = 0$ or path homotopic into a copy of $\tilde{B}$ if $\epsilon_n = 1$ or $\tilde{A}$ if $\epsilon_n = -1$. In the second case, since $c_n = \eta_{\tilde{B}}(o(c_n), t(c_n))$ and $c_n' = \eta_{\tilde{B}}(o(c_n'), t(c_n'))$, the definition of relative normal forms tells us that $c_n*\bar{c}_n'$ lies in $\tilde{B}$ after removing backtracking. Hence $c_n = c_n'$ otherwise the first edge of $c_n$ or $c_n'$ would form the corner of a vertical square together with $t^{\epsilon_n}$. The first case is handled similarly and so $c$ and $c'$ are equal as paths.

%%%%%%%%%%%%%%%%%%%%
\subsection{Quasi-geodesic normal forms for graphs of hyperbolic spaces}
%%%%%%%%%%%%%%%%%%%%

The definition of quasi-geodesics may be generalised in the following way. Let $X$ be a metric space and $f:\R_{\geq0}\to \R_{\geq 0}$ be a monotonic increasing function. A path $c:I\to X$, parametrised by arc length, is an \emph{$f$-quasi-geodesic} if, for all $0\leq p \leq q\leq \abs{c}$, the following is satisfied:
\[
q - p\leq f(d(c(p), c(q)))\cdot d(c(p), c(q)) + f(d(c(p), c(q)))
\]
If $f$ is bounded above by a constant $K$, then $c$ is simply a $K$-quasi-geodesic.

\begin{theorem}
\label{subexponential_distortion}
Let $X$ be a geodesic hyperbolic metric space and $f:\R_{\geq 0}\to \R_{\geq0}$ a subexponential function. There is a constant $K(f)$ such that all $f$-quasi-geodesics are $K(f)$-quasi-geodesics.
\end{theorem}

\begin{proof}
Follows from \cite[Corollary 7.1.B]{gromov_87}. Alternatively, a slight modification of the proof of the Morse lemma, see for example \cite[Proposition 3.3]{alonso_91}, replacing quasi-geodesics with $f$-quasi-geodesics, yields that there is a constant $K'(f)$ such that $f$-quasi-geodesics and geodesics lie in the $K'$-neighbourhoods of each other. So now let $\gamma$ be a geodesic and let $c$ be an $f$-quasi-geodesic with the same endpoints. For each positive integer $i\leqslant |\gamma|$, there is some $j_i\leqslant |c|$ such that $d(\gamma(i), c(j_i))\leqslant K'$ and so $d(c(j_i), c(j_{i+1}))\leqslant 2K' + 1$. Hence, we have $j_{i+1} - j_i\leqslant f(2K' + 1)\cdot (2K'+1) + f(2K'+1)$ for all $i\leqslant |\gamma|$. Thus, setting $K(f) = f(2K'+1)\cdot (2K'+2)$ yields that $|c|\leqslant K(f)\cdot d(o(c), t(c))$ and thus, the result.
\end{proof}

It follows from Theorem \ref{subexponential_distortion} that there is a gap in the spectrum of possible distortion functions of subgroups of hyperbolic groups. This is known by work of Gromov \cite{gromov_87} and appears as \cite[Proposition 2.1]{kapovich_01}.

\begin{corollary}
Subexponentially distorted subgroups of hyperbolic groups are undistorted and hence, quasi-convex.
\end{corollary}

The following result can be thought of as a strengthening of \cite[Corollary 1.1]{kapovich_01}. We make use of Theorem \ref{subexponential_distortion}.

\begin{theorem}
\label{hyperbolic_normal_forms}
Let $\mathcal{X} = (\Gamma, \{X\}, \{C\}, \{\partial^{\pm}\})$ be a graph of spaces with $\Gamma$ consisting of a single vertex and edge. Let $\eta_{\tilde{A}}:\tilde{X}^{(0)}\times\tilde{X}^{(0)}\to I(\tilde{X})$, $\eta_{\tilde{B}}:\tilde{X}^{(0)}\times\tilde{X}^{(0)}\to I(\tilde{X})$ and $\eta_{\tilde{X}}:\tilde{X}^{(0)}\times\tilde{X}^{(0)}\to I(\tilde{X})$ be $\pi_1(X)$-equivariant normal forms, with $\eta_{\tilde{A}}$ relative to $\tilde{A}$ and $\eta_{\tilde{B}}$ relative to $\tilde{B}$. Suppose that the following are satisfied:
\begin{enumerate}
\item $\tilde{X}$ is hyperbolic,
\item $\pi_1(X_{\mathcal{X}})$ acts acylindrically on the Bass-Serre tree $T$,
\item $\eta_{\tilde{A}}$, $\eta_{\tilde{B}}$ and $\eta_{\tilde{X}}$ are prefix-closed quasi-geodesic normal forms.
\end{enumerate}
Then the graph of spaces normal form induced by $\left(\{\eta_{\tilde{X}}\}, \{\eta_{\tilde{A}}, \eta_{\tilde{B}}\}\right)$:
\[
\eta:\tilde{X}^{(0)}_{\mathcal{X}}\times\tilde{X}^{(0)}_{\mathcal{X}}\to I(\tilde{X}_{\mathcal{X}})
\]
is prefix-closed and quasi-geodesic.
\end{theorem}

Under the same hypothesis, Kapovich showed that for each pair of points in $X_{\mathcal{X}}$, there exists a reduced quasi-geodesic path connecting them. Theorem \ref{hyperbolic_normal_forms} shows that graph of spaces normal forms provide us with such paths. This fact will be key for the proof of our main theorems.

\begin{proof}[Proof of Theorem \ref{hyperbolic_normal_forms}]
The fact that $\eta$ is prefix-closed is Theorem \ref{gos_normal_form}. By Theorem \ref{kapo_main}, we have that $\tilde{X}_{\mathcal{X}}$ is $\delta$-hyperbolic and $\tilde{X}\injects X_{\tilde{\mathcal{X}}}$ is a quasi-isometric embedding. Thus, there is a constant $K\geq 0$ such that the images of $\eta_{\tilde{A}}$, $\eta_{\tilde{B}}$ and $\eta_{\tilde{X}}$ are $K$-quasi-geodesic in $\tilde{X}_{\mathcal{X}}$. Now let $x, y\in \tilde{X}_{\mathcal{X}}^{(0)}$ be any two points and $\gamma$ a geodesic connecting them through the 1-skeleton. We may factorise this geodesic
\[
\gamma = \gamma_0*t^{\epsilon_1}*\gamma_1*...*t^{\epsilon_k}*\gamma_k
\]
such that $\epsilon_i = \pm 1$, each $\gamma_i$ is path homotopic into a copy of $\tilde{X}$ and there is no subpath $\gamma_i*t^{\epsilon_{i+1}}*...*\gamma_{j}$ with $i<j$ that is path homotopic into some copy of $\tilde{X}$. If $\rho\in I(\tilde{X}_{\mathcal{X}})$, we will write $\tilde{X}_{\rho}$ to denote the copy of $\tilde{X}$ in $\tilde{X}_{\mathcal{X}}$ that $\rho$ ends in. When it makes sense, we will do the same for $\tilde{A}$ and $\tilde{B}$. This is well defined since all the $\pi_1(X_{\mathcal{X}})$ translates of $\tilde{A}$ are disjoint or equal and all the $\pi_1(X_{\mathcal{X}})$ translates of $\tilde{B}$ are disjoint or equal.

Let $p\in \tilde{X}^{(0)}_{\gamma}$, let $\alpha$ be a geodesic connecting $y$ with $p$ and let $\beta$ be a geodesic connecting $o(\gamma_k)$ with $p$. Denoting by $M = M(K, \delta)$ the maximal Hausdorff distance of all $\eta_{\tilde{X}}$, $\eta_{\tilde{A}}$ and $\eta_{\tilde{B}}$ normal forms from their respective geodesics in $\tilde{X}_{\mathcal{X}}$ and by $L = \delta + M$, we claim that 
\[
\abs{\eta(x, p)} \leq K(k+1)(kL + |\gamma| + |\alpha| + k+1).
\]

The proof of the claim is by induction on $k$. First suppose that $k = 0$. We will write $\simeq$ to mean path homotopic. Since $\beta \simeq \gamma_k* \alpha$ and $o(\beta), t(\beta)\in \tilde{X}_{\gamma}$, we get that $\eta(x, p) = \eta_{\tilde{X}}(x, p)$ is actually a $K$-quasi-geodesic path. This establishes the base case.

Now suppose it is true for all $i<k$. Suppose $\epsilon_k = -1$, the other case is the same. Let $\beta' = \eta_{\tilde{A}}(o(\gamma_k), p) = a* c$ where $a\subset \tilde{A}_{\gamma*\bar{\gamma}_k}$ and the first edge of $c$ is not in $\tilde{A}_{\gamma*\bar{\gamma}_k}$. Since $a$ is supported in a copy of $\tilde{A}$, the path $t^{-1}*a$ may be homotoped through vertical squares to a path $b*t^{-1}$, where $b$ is supported in $\tilde{B}_{\gamma*\bar{\gamma}_k*t}$. Let $\alpha'$ be a geodesic connecting $o(b)$ to $t(b)$. Since $a*c = \eta_{\tilde{A}}(o(a), t(c))$, the definition of relative normal forms implies that $c = \eta_{\tilde{A}}(o(c), t(c))$. Let $t(\alpha') = p'$. By uniqueness and the definition of graph of space normal forms, we have 
\begin{align}
\label{eq1}
\eta(x, p) &= \eta(x, p')* t^{-1}*c.
\end{align}
Letting $\gamma' = \gamma_0* t^{\epsilon_1}*...*\gamma_{k-1}$, the inductive hypothesis applied to $\gamma'$ and $\alpha'$ yields that
\begin{align}
\label{ineq1}
\abs{\eta(x, p')} &\leq Kk((k-1)L + |\gamma'| + |\alpha'| + k).
\end{align}
Applying (\ref{ineq1}) to (\ref{eq1}), we obtain
\begin{align}
\label{ineq2}
\abs{\eta(x,q)} &\leq Kk((k-1)L + \abs{\gamma'} + \abs{\alpha'} + k) + \abs{c} + 1
\end{align}
By assumption, $\beta'$ is a $K$-quasi-geodesic. Hence there is a point $q\in \beta$ such that $d(t(a), q)\leq M$. But then since $\beta\cup \alpha\cup \gamma_k$ forms a geodesic triangle, there is a point $z\in \alpha\cup \gamma_k$ such that $d(t(a), z)\leq L$. We may now divide into two subcases: either $z\in \alpha$ or $z\in \gamma_k$. 

\begin{figure}
\centering
\includegraphics[scale=0.7]{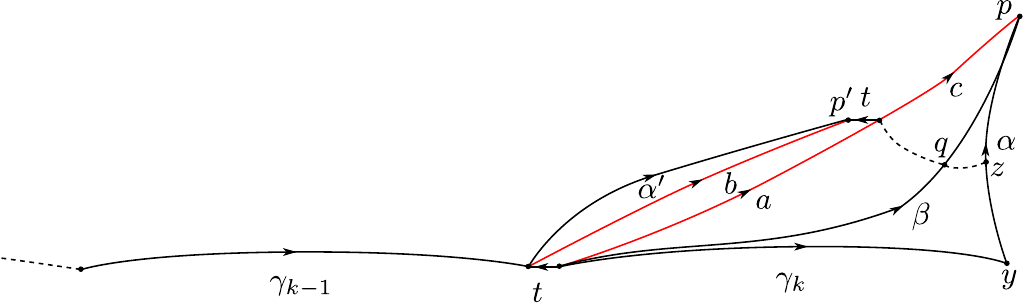}
\caption{Case 1: $z\in \alpha$.}
\label{case_1}
\end{figure}

First suppose $z$ is a vertex traversed by $\alpha$, see Figure \ref{case_1}. It follows that 
\begin{align*}
d(o(a), t(a)) &\leq \abs{\gamma_k} + L + \abs{\alpha}\\
d(o(c), t(c)) & \leq L + \abs{\alpha}
\end{align*}
and hence, since $c$ is a $K$-quasi-geodesic, we have
\begin{align}
\label{ineq3}
\abs{c} &\leq K(L + \abs{\alpha}) +K.
\end{align}
We also have 
\begin{align}
\label{ineq4}
\abs{\alpha'}&\leq d(o(a), t(a)) + 2 \leq \abs{\gamma_k} + L + \abs{\alpha} + 2.
\end{align}
In order, applying (\ref{ineq3}), (\ref{ineq4}) and the fact that $\abs{\gamma} = \abs{\gamma'} + 1 + \abs{\gamma_k}$ to (\ref{ineq2}), we obtain
\begin{align*}
\abs{\eta(x,q)} &\leq Kk((k-1)L + \abs{\gamma'} + \abs{\alpha'} + k) + K(L + \abs{\alpha}) + K + 1\\
		&\leq Kk((k-1)L + \abs{\gamma'} + \abs{\gamma_k} + L + \abs{\alpha} + k + 2) + K(L + \abs{\alpha}) + K + 1\\
		&\leq Kk(kL + \abs{\gamma} + \abs{\alpha} + k + 1) + K(L + \abs{\alpha}) + K + 1\\
		&\leq Kk(kL + \abs{\gamma} + \abs{\alpha} + k + 1) + K(L + \abs{\alpha}+ 2)\\
		&\leq K(k+1)(kL + \abs{\gamma} + \abs{\alpha} + k + 1).
\end{align*}

\begin{figure}
\centering
\includegraphics[scale=0.7]{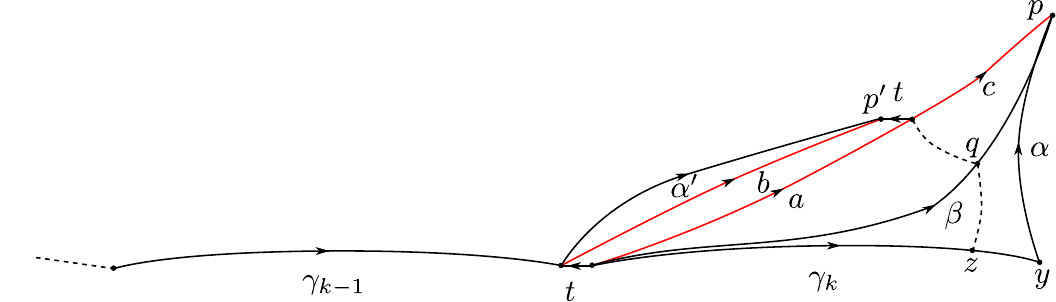}
\caption{Case 2: $z\in \gamma_k$.}
\label{case_2}
\end{figure}
Now we deal with the other case, see Figure \ref{case_2}. If $z$ is a vertex traversed by $\gamma_k$, then 
\begin{align*}
d(o(a), t(a)) &\leq d(o(\gamma_k), z) + L\\
d(o(c), t(c)) &\leq L + d(z, y) + \abs{\alpha}
\end{align*}
and hence, since $c$ is a $K$-quasi-geodesic, we have
\begin{align}
\label{ineq5}
\abs{c} &\leq K(L + d(z, y) + \abs{\alpha})+ K.
\end{align}
We have 
\begin{align}
\label{ineq6}
\abs{\alpha'} &\leq d(o(a), t(a)) + 2 \leq d(o(\gamma_k), z) + L + 2.
\end{align}
In order, applying (\ref{ineq5}) and (\ref{ineq6}) to (\ref{ineq1}) we obtain
\begin{align*}
\abs{\eta(x, p)} \leq& Kk((k-1) L + \abs{\gamma'} + \abs{\alpha'} + k)\\
& + K(L + d(z, q) + \abs{\alpha}) + K +1\\
\leq& Kk((k-1) L + \abs{\gamma'} + d(o(\gamma_k), z) + L + k + 2)\\
& + K(L + d(z, q) + \abs{\alpha}) + K +1\\
\leq& Kk(kL + \abs{\gamma'} + d(o(\gamma_k), z) + k + 2)\\
&+ K(L + d(z, q) + \abs{\alpha} + 2)
\end{align*}
and then using the fact that $|\gamma| = |\gamma'| + 1 + d(o(\gamma_k), z) + d(z, q)$, we obtain
\begin{align*}
\abs{\eta(x, p)} \leq& Kk(kL + \abs{\gamma} + \abs{\alpha} + k + 1) + K(L + \abs{\gamma} + \abs{\alpha} + k + 1)\\
\leq& K(k+1)(kL + \abs{\gamma} + \abs{\alpha} + k + 1)
\end{align*}
This concludes the proof of the claim.

Consider the polynomial function $f:\N\to \R$ given by
\[
f(n) = K(n+1)^{2}(L + 2).
\]
We have shown that for all $x, y\in \tilde{X}_{\mathcal{X}}^{(0)}$, we have $\abs{\eta(x, y)}\leq f(d(x, y))$. We now want to show that the graph of spaces normal forms are actually $f$-quasi-geodesics. 

Let $h = h_0* t^{\epsilon_1}* h_1*...* t^{\epsilon_k}* h_k\in \Ima(\eta)$ be a normal form and let $h' = h_i'* t^{\epsilon_{i+1}}* h_{i+1}*...* t^{\epsilon_j}* h_j'$ be a subpath of $h$. Then we have that
\[
\eta_{\tilde{X}}(o(h_i'), t(h_i'))*t^{\epsilon_{i+1}}* h_{i+1}*...* t^{\epsilon_j}* h_j'
\]
satisfies the three conditions of the definition of a graph of space normal form, where here we are using the fact that $\eta_{\tilde{A}}$ and $\eta_{\tilde{B}}$ are prefix closed to see that $h_j'$ is a normal form. Combining this with the uniqueness of graph of space normal forms we have
\[
\eta(o(h'), t(h')) = \eta_{\tilde{X}}(o(h_i'), t(h_i'))*t^{\epsilon_{i+1}}* h_{i+1}*...* t^{\epsilon_j}* h_j'.
\] 
We showed that:
\[
\abs{\eta(t(h'_i), t(h'))}\leq f(d(t(h_i'), t(h')))
\]
and so we know that 
\begin{align*}
|h'| &= \abs{h_i'} + \abs{\eta(t(h_i'), t(h'))}\\
	&\leq K d(o(h_i'), t(h_i')) + K + \abs{\eta(t(h_i'), t(h'))}\\
	&\leq K d(o(h_i'), t(h_i')) + K + f(d(t(h_i'), t(h')))\\
	&\leq f(d(o(h_i'), t(h_i'))) + f(d(t(h_i'), t(h')))\\
	&\leq f(d(o(h'), t(h'))).
\end{align*}
By Theorem \ref{subexponential_distortion}, it follows that there is some constant $K' = K'(f)$ such that $\eta$ is a $K'$-quasi-geodesic normal form.
\end{proof}

Theorem \ref{hyperbolic_normal_forms} can be used to show that certain subgroups of a hyperbolic acylindrical graph of hyperbolic groups are quasi-convex. In the section that follows, we are going to do precisely this to deduce that Magnus subgroups of one-relator groups with a quasi-convex hierarchy are quasi-convex. The main idea we employ is to use quasi-geodesic normal forms in vertex spaces to build quasi-geodesic normal forms relative to given subgroups. Other conditions of a different flavour already exist in the literature, see \cite{BW13} and the references therein. These results cannot be applied directly as the edge groups in the Magnus splittings are not necessarily malnormal and are usually very far from being Noetherian. Nevertheless, an alternative approach could involve understanding the splittings of Magnus subgroups of one-relator groups induced by the Magnus hierarchy and then following the approach of Dahmani from \cite{Da03} for showing local (relative) quasi-convexity of limit groups. We opt for our more direct approach which may be of independent interest.

\subsection{Quasi-convex Magnus subgraphs}

Let us first define what a quasi-convex one-relator hierarchy is. This definition is a reformulation of Wise's \cite{wise_21_quasiconvex} notion.

\begin{definition}
A one-relator tower (respectively, hierarchy) $X_N\immerses ...\immerses X_1\immerses X_0$ is a quasi-convex tower (respectively, hierarchy) if $\pi_1(A_{i+1}), \pi_1(B_{i+1})$ are quasi-isometrically embedded in $\pi_1(X_i)$ for all $i$, where $A_{i+1}, B_{i+1}\subset X_{i+1}$ are the associated Magnus subgraphs.
\end{definition}

In order to prove that one-relator groups with torsion have quasi-convex (one-relator) hierarchies, Wise showed in \cite[Lemma 19.8]{wise_21_quasiconvex} that their Magnus subgroups are quasi-convex. When the torsion assumption is dropped, this is certainly no longer true. However, under additional hypotheses, we may recover quasi-convexity of Magnus subgroups. The only results specifically from the theory of one-relator groups we shall need are the Freiheitssatz and our hierarchy results.

\begin{theorem}
\label{one-relator_normal_form_inductive}
Let $X = (\Gamma, \lambda)$ be a finite one-relator complex, $p:Y\to X$ a cyclic cover and $Z\in \zd(p)$ a finite one-relator $\Z$-domain. Suppose further that the following hold:
\begin{enumerate}
\item $\pi_1(X)$ is hyperbolic.
\item For all connected Magnus subgraphs $W\subset Z$, $\pi_1(W)$ is quasi-convex in $\pi_1(Z)$.
\item $\pi_1(Z)$ is quasi-convex in $\pi_1(X)$.
\end{enumerate}
Then for all connected Magnus subgraphs $C\subset X$, $\pi_1(C)$ is quasi-convex in $\pi_1(X)$. 
\end{theorem}

\begin{proof}
Let $C\subset X$ be a connected Magnus subgraph. Denote by $Z = (\Lambda, \tilde{\lambda})$. Since quasi-convexity is transitive, we may assume that $E(C) = E(\Gamma) - \{f\}$ for some edge $f$. Note that $C$ is connected and so $f$ is non separating.

For each edge $e\in E(X)$, choose a lift $e_0$ in $Y$ as in Subsection \ref{existence_section} and denote by $e_i = t^i\cdot e_0$ for each $i\in \Z$. Denote by $m_e$ the smallest integer such that $e_{m_e}$ is traversed by $\tilde{\lambda}$, and by $M_e$ the largest. We have two cases to consider.

\begin{mycase}

\case For all $f_i\in E(\Lambda)$, we have that $m_f\leq i\leq M_f$ and $f_i$ is non separating in $\Lambda$.

The action of $\deck(p)$ on $Y$, combined with the Freiheitssatz, gives us a graph of spaces $\mathcal{X} = (S^1, \{Z\}, \{Z\cap t\cdot Z\}, \{\partial^{\pm}\})$ as in Proposition \ref{splitting}. We moreover have a map (the horizontal map):
\[
\mathsf{h}:X_{\mathcal{X}}\to X
\]
that is a homotopy equivalence obtained by collapsing the edge space onto the vertex space as in Proposition \ref{homotopy_equivalence}. This lifts to a map in the universal covers:
\[
\tilde{\mathsf{h}}:X_{\tilde{\mathcal{X}}}\to\tilde{X}.
\]
Denote by $A, B\subset Z$ the Magnus subgraphs that are the images of $\partial^{\pm}$. That is, $A = t^{-1}\cdot(Z\cap t\cdot Z)$ and $B = Z\cap t\cdot Z$. 

Denote by $A' = \Lambda - f_{M_f}$ and $B' = \Lambda - f_{m_f}$. These are connected since $f_{m_f}$ and $f_{M_f}$ are non separating and they are Magnus subgraphs since both $f_{m_f}$ and $f_{M_f}$ are traversed by $\tilde{\lambda}$. In particular, $p^{-1}(C)\cap Z\subset A'\cap B'$. Denote by $\tilde{Z}$ the universal cover of $Z$. By the Freiheitssatz, the universal covers of $A\subset A'$ and $B\subset B'$ are trees, including into $\tilde{Z}$. Denote by $\tilde{A}\subset \tilde{A}'$ and $\tilde{B}\subset \tilde{B}'$ subgraphs of $\tilde{Z}$ corresponding to universal covers of $A\subset A'$ and $B\subset B'$ respectively. Since $\pi_1(A')$ and $\pi_1(B')$ are quasi-convex subgroups of $\pi_1(Z)$, there are prefix-closed, quasi-geodesic, $\pi_1(Z)$-equivariant normal forms $\eta_{\tilde{A}'}$, $\eta_{\tilde{B}'}$ for $\tilde{Z}$, relative to $\tilde{A}'$ and $\tilde{B}'$ respectively. Recall that our normal forms are immersed paths and so $\eta_{\tilde{A}'}$ and $\eta_{\tilde{B}'}$ necessarily restrict to geodesic normal forms on $\tilde{A}'$ and $\tilde{B}'$. Moreover, they necessarily agree on intersections of translates of $\tilde{A}'$ with translates of $\tilde{B}'$. Since $\tilde{A}$ and $\tilde{B}$ are subtrees of $\tilde{A}'$ and $\tilde{B}'$ respectively, these are also normal forms relative to $\tilde{A}$ and $\tilde{B}$. Hence, we may apply Theorem \ref{gos_normal_form} to obtain unique prefix-closed $\pi_1(X_{\mathcal{X}})$-equivariant graph of space normal forms $\eta$ for $X_{\tilde{\mathcal{X}}}$, induced by $\left(\{\eta_{\tilde{A}'}\}, \{\eta_{\tilde{A}'}, \eta_{\tilde{B}'}\}\right)$. By Theorems \ref{kapo_main} and \ref{hyperbolic_normal_forms}, $\eta$ is also quasi-geodesic.

Now let $c:I\immerses \tilde{C}\subset \tilde{X}$ be an immersed path and let $c':I\immerses X_{\tilde{\mathcal{X}}}$ be a path such that $\mathsf{h}\circ c'$ is path homotopic to $c$ via a sequence of collapses of segments in the domain which mapped to vertical edges in $X_{\tilde{\mathcal{X}}}$ under $c'$. Since $X_{\mathcal{X}}$ is finite, there is a constant $k$ such that preimages of vertices in $\tilde{X}$ under $\mathsf{h}$ are segments of length at most $k$. In particular, $|c'|\leqslant k|c|$. After possibly performing finitely many homotopies through vertical squares (which do not alter the composition with $\mathsf{h}$), we may assume that $c' = c_0*t^{\epsilon_1}*c_1*...*t^{\epsilon_n}*c_n$ has the property that for all $i$, if $e$ is the first edge that $c_i$ traverses, then $t^{\epsilon_i}*e$ does not form the corner of a vertical square. Since $p^{-1}(C)\cap Z\subset A'\cap B'$, we see that each $c_i$ is both in normal form with respect to $\eta_{\tilde{A}'}$ and $\eta_{\tilde{B}'}$ by the remarks in the previous paragraph. Hence, $c'$ is in normal form with respect to $\eta$. Since $\eta$ was a quasi-geodesic normal form, it follows that $\pi_1(C)$ is quasi-convex in $\pi_1(X)$.

\case There is some $f_i\in E(\Lambda)$ such that either $i<m_f$, or $i>M_f$, or $f_i$ is separating in $\Lambda$. 

Let $X'$ be the one-relator complex obtained from $X$ by adding a single edge, $d$. Thus, we have $\pi_1(X') = \pi_1(X)*\langle x\rangle$. If $\phi:\pi_1(X)\to \Z$ is the epimorphism inducing $p$, denote by $\phi':\pi_1(X')\to \Z$ the epimorphism such that $\phi'\mid\pi_1(X) = \phi$ and $\phi'(x) = 1$. Let $Z''\subset Z$ be the subcomplex obtained from $Z$ by removing each $f_i$ with $i<m_f$ or $i>M_f$. If $p':Y'\to X'$ is the cyclic cover induced by $\phi'$, then $Z''$ lifts to $Y'$. Moreover, for appropriately chosen integers $k\leq l$, we see that 
\[
(\Lambda', \tilde{\lambda}') = Z' = Z''\cup\left(\bigcup_{i = k}^ld_i\right)
\]
is a one-relator $\Z$-domain for $p'$. By construction, for all $f_i\in E(\Lambda')$, we have that $m_f\leq i\leq M_f$ and that $f_i$ is non separating in $\Lambda'$. Moreover
\[
\pi_1(Z') = \bigast_{i=0}^{k-l}\pi_1(A_i)^{x^i}
\]
where each $A_i$ is a (possible empty) connected subcomplex of $Z$. By hypothesis, $\pi_1(A')$ is quasi-convex in $\pi_1(X')$ for all connected subgraphs $A'\subset Z'$. By the proof of the first case, we see that $\pi_1(C)$ is quasi-convex in $\pi_1(X')$. But then $\pi_1(C)$ must also be quasi-convex in $\pi_1(X)$ and we are done.\qedhere
\end{mycase}
\end{proof}

\begin{corollary}
\label{Magnus_quasiconvex}
Let $X = (\Gamma, \lambda)$ be a one-relator complex. Suppose that $\pi_1(X)$ is hyperbolic and $X$ admits a quasi-convex one-relator hierarchy. Then $\pi_1(A)$ is quasi-convex in $\pi_1(X)$ for all Magnus subgraphs $A\subset X$.
\end{corollary}

\begin{proof}
The proof is by induction on hierarchy length. The base case is clear. Theorems \ref{kapo_main} and \ref{one-relator_normal_form_inductive} handle the inductive step.
\end{proof}

At this point we remark that we have essentially proved the equivalence between (\ref{itm:hqch}) and (\ref{itm:acylindrical_h}) from Theorem \ref{main}. Nevertheless, we postpone the complete proof to Section \ref{sec:main}.

\section{$\Z$-stable HNN-extensions of one-relator groups}
\label{sec:stable}

Let us repeat the definition of the \emph{$\Z$-stable number} $s\Z(\psi)$ from the introduction. If $\psi:A\to B$ is an isomorphism of two subgroups $A, B<H$, inductively define 
\[
\mathcal{A}^{\psi}_0 = \{[H]\}, \quad \ldots, \quad \mathcal{A}^{\psi}_{i+1} = \{[\psi(A\cap A_i)]\}_{A_i\in [A_i]\in \mathcal{A}^{\psi}_i}, \quad \ldots
\]
Then we denote by $\bar{\mathcal{A}}_i^{\psi}\subset \mathcal{A}_i^{\psi}$ the subset corresponding to the conjugacy classes of non-cyclic subgroups. Define the \emph{$\Z$-stable number} $s\Z(\psi)$ of $\psi$ as
\[
s\Z(\psi) = \sup{\{k + 1 \mid \bar{\mathcal{A}}_k^{\psi}\neq\emptyset\}}\in \N\cup\{\infty\}
\]
where $s\Z(\psi) = \infty$ if $\bar{\mathcal{A}}_i^{\psi}\neq \emptyset$ for all $i$.  We say that $H*_{\psi}$ is \emph{$\Z$-stable} if $s\Z(\psi)<\infty$.

For the purposes of this section, unless otherwise stated, we will always assume that $H$ is a finitely generated one-relator group and that $A, B<H$ are two strongly inert subgroups of Magnus subgroups for some one-relator presentation of $H$. For the main results of this article, the reader should have in mind the case in which $A$ and $B$ are Magnus subgroups themselves. Let $\psi:A\to B$ be an isomorphism. Consider the HNN-extension of $H$ over $\psi$:
\begin{equation*}
G \isom H*_{\psi} \isom \langle H, t \mid tat^{-1} = \psi(a), \forall a\in A\rangle
\end{equation*}
We will call this an \emph{inertial one-relator extension}.

\begin{remark}
\label{splitting_remark}
Recall that if $X$ is a one-relator complex, then a one-relator splitting of $\pi_1(X)$ is a HNN-splitting of $\pi_1(X)$ as in Theorem \ref{new_hierarchy}. By Remark \ref{magnus_remark}, every one-relator splitting $\pi_1(Z)*_{\psi}$ of $\pi_1(X)$ is an inertial one-relator extension.
\end{remark}

\begin{remark}
If $H$ is a one-relator group with the trivial relation, $H$ itself is a Magnus subgroup of $H$. Thus, HNN-extensions of free groups with finitely generated strongly inert edge groups are inertial one-relator extensions.
\end{remark}

\begin{remark}
For simplicity, we focus only on HNN-extensions. However, with minor modifications, most of the results in this section hold for graphs of groups $\{\Gamma, \{G_v\}, \{G_e\}, \{\partial_e^{\pm}\}\}$ satisfying the following:
\begin{enumerate}
\item $\Gamma$ is a connected graph with $\chi(\Gamma)\geq 0$,
\item for each $v\in V(\Gamma)$, $G_v$ has a one-relator presentation so that each adjacent edge group is a strongly inert subgroup of a Magnus subgroup of $G_v$.
\end{enumerate}
This class of groups includes groups with staggered presentations.
\end{remark}

Let $T$ denote the Bass--Serre tree associated with the inertial one-relator extension $H*_{\psi}$. See \cite{serre_80} for the relevant notions in Bass--Serre theory. We will identify each vertex of $T$ with a left coset of $H$. Denote by 
\begin{align*}
\mathcal{S}^n &= \{S \mid S \text{ a geodesic segment of length $n$ with an endpoint at $H$}\}\\
\mathcal{S} &= \bigcup_i\mathcal{S}^i
\end{align*}
Each edge in $T$ has an orientation induced by $\psi$. Denote by $\mathcal{S}^n_{+}\subset \mathcal{S}^n$ the geodesic segments which only consist of edges pointing away from $H$ and by $\mathcal{S}^n_{-}\subset \mathcal{S}^n$ those pointing towards $H$. The elements $[A_n]\in \bar{\mathcal{A}}_n^{\psi}$ correspond to stabilisers of segments $S\in \mathcal{S}_{-}^n$ such that $\rr(\stab(S))\neq 0$.

For inertial one-relator extensions, the ranks of stabilisers of elements in $\mathcal{S}^n$ are bounded in a strong sense.

\begin{lemma}
\label{stab_ranks}
For all $n\geq 1$, the following holds:
\begin{align*}
\sum_{\substack{H\cdot S \\ S\in \mathcal{S}^n}}\rr(\stab(S)) &= \sum_{\substack{H\cdot S \\ S\in \mathcal{S}_{-}^n}}\rr(\stab(S)) + \sum_{\substack{H\cdot S \\ S\in \mathcal{S}_{+}^n}}\rr(\stab(S))\\
\sum_{\substack{H\cdot S \\ S\in \mathcal{S}_{\pm}^n}}\rr(\stab(S)) &\leq\rr(A).
\end{align*}
\end{lemma}

\begin{proof}
Let $S\in \mathcal{S}^n$, then $\stab(S) = H\cap H^{c^{-1}}$ where $c$ is equal to a reduced word of the form:
\[
c = c_1t^{\epsilon_1}\ldots c_nt^{\epsilon_n}
\]
with $\epsilon_i = \pm 1$ and $c_i\in H$. The segment $S$ is the geodesic segment with vertices $H, c_1t^{\epsilon_1}H, \ldots, cH$. By assumption we have that $A$ is a strongly inert subgroup of a Magnus subgroup $A'<H$. We have that $\rr(A\cap A^a) = 0$ for all $a\in A' - A$ by Remark \ref{cyclonormal_remark} (similarly for $B$). Combining this with Theorem \ref{conjugates_intersection}, if there is some $i$ such that $\epsilon_i = -\epsilon_{i+1}$, then $\stab(S)$ is either cyclic or trivial since our word was reduced. Denote by $W_{+1}^n$ the set of reduced words of the form $tc_2t\ldots c_{n}t$ and by $W_{-1}^n$ the set of reduced words of the form $t^{-1}c_2t^{-1}\ldots c_{n}t^{-1}$. Let $T_{\pm1}^n$ be a set of representatives for the set of double cosets $H\backslash W_{\pm1}^n/H$. Then we have
\[
\sum_{\substack{H\cdot S \\ S\in \mathcal{S}^n}}\rr(\stab(S)) = \sum_{c\in T_{+1}^n}\rr\left(H\cap H^{c^{-1}}\right) + \sum_{c\in T_{-1}^n}\rr\left(H\cap H^{c^{-1}}\right)
\]
which yields the equality from the statement. By symmetry, it suffices to only bound one of the sums above to establish the inequality. We may assume that each word in $T_{+1}^n$ is a prefix of a word in $T_{+1}^{n+1}$. If $c\in T_{+1}^n$, denote by $T_c$ the set of elements $c_{n+1}\in H$ such that $cc_{n+1}t\in T_{+1}^{n+1}$. We first claim that each element $c_{n+1}\in T_c$ is in a distinct $(H\cap H^c)c_{n+1}B$ double coset. Let $c_{n+1}, c_{n+1}'\in T_{c}$ be such that $c_{n+1}' = hc_{n+1}b$ where $h = c^{-1}h'c\in H\cap H^c$ and $b\in B$. Then we have that $cc_{n+1}'t = h'cc_{n+1}t\psi^{-1}(b)\in Hcc_{n+1}tH$, a contradiction. We next claim that
\[
\sum_{c\in T_{+1}^n}\rr\left(H\cap H^{c^{-1}}\right) = \sum_{c\in T_{+1}^n}\rr(H\cap H^c)\leqslant \rr(A).
\]
The equality is obvious, we now show the inequality. The proof of the claim proceeds by induction with the base case being clear as $H\cap H^{t^{-1}} = A$. So let us now assume the inductive hypothesis. We have
\begin{align*}
\sum_{c\in T_{+1}^{n+1}}\rr(H\cap H^c) &= \sum_{c\in T_{+1}^n}\sum_{c_{n+1}\in T_c}\rr\left(H\cap H^{cc_{n+1}t}\right) \\
			&\leqslant \sum_{c\in T_{+1}^n}\sum_{\substack{(H\cap H^c)c_{n+1}B\\ c_{n+1}\in H}}\rr\left(H^{t^{-1}}\cap H\cap H^{cc_{n+1}}\right)\\
			& =  \sum_{c\in T_{+1}^n}\sum_{\substack{(H\cap H^c)c_{n+1}B\\ c_{n+1}\in H}}\rr\left( B\cap (H\cap H^c)^{c_{n+1}}\right)\\
			&\leqslant \sum_{c\in T^n_{+1}}\rr(H\cap H^c)\\
			&\leqslant \rr(A)
\end{align*}
where the first inequality follows from our previous claim, the second inequality follows from Theorem \ref{subgroup_Magnus_intersection} and where the final inequality follows from the inductive hypothesis. The lemma now readily follows.
\end{proof}

The following proposition is key in our proof of $\Z$-stability of one-relator hierarchies of one-relator groups with negative immersions. The main consequence is that if $s\Z(\psi) = \infty$, then there exists some element $g\in G$ acting hyperbolically on $T$ such that $\rr(H\cap H^{g^n})\neq 0$ for all $n$.

\begin{proposition}
\label{biinfinite_stab}
If $s\Z(\psi) = \infty$, then there are $1\leq n\leq\rr(A)$ many $H$-orbits of biinfinite geodesics $S\subset T$ such that the following holds:
\begin{enumerate}
\item $S$ contains the vertex $H$.
\item Every finite subset of $S$ has non-cyclic, non-trivial stabiliser.
\end{enumerate}
Moreover, for every such biinfinite geodesic, the following holds:
\begin{enumerate}
\item Every edge in $S$ is directed in the same direction.
\item There exists an element $g\in G$ acting hyperbolically on $T$ with translation length at most $\rr(A)$ such that $g\cdot S = S$.
\end{enumerate}
\end{proposition}

\begin{proof}
The fact that every such geodesic must consist of edges directed in the same direction follows from the equality in Lemma \ref{stab_ranks}. The fact that there are $1\leq n\leq \rr(A)$ many such $H$-orbits of geodesics follows by applying the pigeonhole principle and the inequality from Lemma \ref{stab_ranks}. So now let $S_1, ..., S_n$ be a collection of $H$-orbit representatives of such biinfinite geodesics in $T$. Identify each vertex of $S_i$ with an integer so that $H$ is the vertex associated with $0$. Let $g_{i, j}\in G$ be an element such that $g_{i, j}H$ is the $j^{\text{th}}$ vertex in $S_i$. Then $g_{i, j}^{-1}\cdot S_i$ must be in the same $H$-orbit of some $S_m$. But then by the pigeonhole principle, $S_i, g_{i, 1}^{-1}\cdot S_i, ..., g_{i, \rr(A)}^{-1}\cdot S_{i}$ must contain two biinfinite geodesics in the same $H$-orbit. Suppose that $h_1g_{i, k}^{-1}\cdot S_i = h_2g_{i, l}^{-1}\cdot S_i = S_l$ for some $h_1, h_2\in H$ and $1\leqslant k< l\leqslant \rr(A)$. Then we have
\[
S_i = h_2g_{i, l}(h_1g_{i, k})^{-1}\cdot S_i
\]
where $h_2g_{i, l}(h_1g_{i, k})^{-1}$ acts hyperbolically on $T$ with translation length at most $\rr(A)$.
\end{proof}

\subsection{One-relator groups with torsion}

The following lemma allows one to easily establish $\Z$-stability of a one-relator splitting in certain cases: if the edge groups of a one-relator splitting have no exceptional intersection, then the splitting is $\Z$-stable.

\begin{lemma}
\label{exceptional_stable}
Let $X$ be a one-relator complex and let $\pi_1(X)\isom \pi_1(X_1)*_{\psi}$ be a one-relator splitting where $A, B\subset X_1$ are the associated Magnus subgraphs. Let $A', B'\subset X_1$ be Magnus subgraphs such that $A\subset A'$, $B\subset B'$ and $A'\cap B'$ is connected. If $\pi_1(A')\cap \pi_1(B') = \pi_1(A'\cap B')$, then $s\Z(\psi)<\infty$.
\end{lemma}

\begin{proof}
Let $Y\to X$ be the cyclic cover containing $X_1$ and let $t$ be a generator of the deck group. Let $\gamma:\Gamma\immerses A$ be a graph immersion with $\chi(\Gamma)\leq -1$. By assumption and Theorem \ref{conjugates_intersection}, $\gamma$ is homotopic in $X_1$ to a graph immersion $\gamma_1:\Gamma\immerses B = t\cdot A$ if and and only if $\gamma$ is homotopic in $A$ to a graph immersion $\gamma_1:\Gamma_1\immerses A\cap B$. If $s\Z(\psi)\geq 2$, there is a graph immersion $\gamma$ as above such that $\gamma_1:\Gamma_1\immerses A\cap B\injects t\cdot A$ is homotopic in $t\cdot X_1$ to a graph immersion $\gamma_2:\Gamma_2\immerses A\cap t\cdot A\cap t^2\cdot A\injects 1\cdot (A\cap B)$. Carrying on this argument, we see that if $s\Z(\psi) \geq n$, then there exists a graph immersion $\gamma:\Gamma\immerses A$ such that $\chi(\Gamma)\leq -1$ and such that $\gamma$ is homotopic to a graph immersion $\gamma_n:\Gamma_n\immerses A\cap t\cdot A\cap ...\cap t^n\cdot A$. However, clearly for large enough $n$ we have that $A\cap t\cdot A\cap... \cap t^n\cdot A = \emptyset$.
\end{proof}

The condition from Lemma \ref{exceptional_stable} always holds for one-relator splittings of one-relator groups with torsion by Remark \ref{splitting_remark} and Theorem \ref{exceptional_intersection} (see Remark \ref{torsion_remark}), yielding the following.

\begin{corollary}
\label{torsion_stable}
All one-relator splittings of one-relator groups with torsion are $\Z$-stable.
\end{corollary}

\subsection{One-relator groups with negative immersions}

The aim of this subsection is to show that every one-relator splitting of a one-relator group with negative immersions is $\Z$-stable. In order to do this, we show that certain constraints on the subgroups of $H*_{\psi}$ that are not conjugate into $H$ imply $\Z$-stability.

\begin{theorem}
\label{descending_stable}
If every descending chain of non-cyclic, freely indecomposable proper subgroups of bounded rank
\[
G=H_0>H_1>...>H_n>...
\]
is either finite or eventually conjugate into $H$, then $s\Z(\psi)<\infty$.
\end{theorem}

\begin{proof}
Suppose for contradiction that $s\Z(\psi) = \infty$. Let $S$ be a biinfinite geodesic as in Proposition \ref{biinfinite_stab}. Let $g\in G$ be the element acting by translations on $S$. Since every edge in $S$ points in the same direction, we have that $\phi(g)\neq 0$. Hence, by definition of $S$, $H\cap H^{g^n}$ is non-trivial and non-cyclic for all $n$.

Denote by $\phi:G\to G/\normal{H}\isom \Z$. Consider the descending chain of subgroups 
\[
G = \langle H, g\rangle>\langle H, g^2\rangle>...>\langle H, g^{2^n}\rangle>...
\]
where we denote by $H_n = \langle H, g^{2^n}\rangle$. Since $\rk(H_n)\leq\rk(H) +1$, the ranks are bounded. The chain must be proper since $\phi(H_i) = \langle \phi(g)2^i\rangle \neq \langle \phi(g)2^j\rangle = \phi(H_j)$ for all $i\neq j$. 

Each $H_i$ has a Grushko decomposition
\[
H_i = F(X_i)*J_{i, 1}*...*J_{i, n_i}
\]
that is unique up to permutation and conjugation of factors and where each $J_{i, j}$ is non-cyclic and freely indecomposable. By the Kurosh subgroup theorem, each $J_{i, j}$ is a conjugate of a subgroup of some $J_{i-1, l}$. Furthermore, we have that $\abs{X_i} + 2n_i\leq \rk(H) + 1$.

Now the claim is that for some integer $m\geq 0$ and for all $i>m$, each $J_{i, j}$ is either a conjugate of some $J_{i-1, l}$ or is conjugate into $H$. If this was not the case, then there would be infinite sequences of integers $i_0, i_1, i_2, ...$ and $j_0, j_1, j_2, ...$ and elements $g_1, g_2, ...\in G$ such that
\[
J_{i_0, j_0} > J_{i_1, j_1}^{g_1} > J_{i_2, j_2}^{g_2} > ...
\]
where the inclusions are all proper and no $J_{i_k, j_k}$ is conjugate into $H$. But this contradicts the hypothesis and so the claim is proven.

So now denote by $J_i = J_{i,1}*...*J_{i,n_i}$. For all $i\geq m$, we have that $\phi(J_i) = \langle l\rangle$ for some fixed $l$. But now since $\phi(H_i) = \langle \phi(g)2^i\rangle$, this forces $J_i< \ker(\phi)$. Thus, the induced homomorphism $\phi\mid_{H_m}$ factors through the projection $H_m\to H_m/\normal{J_m}\isom F(X_m)$. But now $\phi\left(g^{2^m}\right)$ is a generator of $\phi(H_m)$, so there is some primitive element $x_m\in F(X_m)$ that maps to a generator of $\phi(H_m)$ and such that $x_m = k_mg^{2^m}$ for some $k_m\in \ker(\phi\mid_{H_m})$. Since $x_m$ is primitive in $F(X_m)$, we have:
\[
H_m = \langle x_m\rangle *K_m = \left\langle k_mg^{2^m}\right\rangle * K_m
\]
for some subgroup $K_m<H_m$ such that $\normal{K_m}=\ker(\phi)$.

Now let $T_m$ be the Bass--Serre tree associated with the HNN-extension $\langle x_m\rangle *K_m$ with trivial edge groups. Each vertex is stabilised by a conjugate of $K_m$ and each edge has trivial stabiliser. Since $H<H_m$, $H<\ker(\phi)$ and $H$ is finitely generated, it follows that $H$ is a subgroup of a free product of finitely many conjugates of $K_m$. Since $g^{2^m}$ acts hyperbolically on $T_m$ and $T_m$ has trivial edge stabilisers, it follows that
\[
H\cap H^{g^{2^{nm}}} = 1
\]
for all $n$ sufficiently large. But then this contradicts the assumption that $s\Z(\psi) = \infty$.
\end{proof}

\begin{corollary}
\label{stable_negative}
All one-relator splittings of one-relator groups with negative immersions are $\Z$-stable.
\end{corollary}

\begin{proof}
The result follows from a result of Louder--Wilton \cite[Theorem B]{louder_21_uniform} and Theorem \ref{descending_stable}.
\end{proof}

\subsection{The graph of cyclic stabilisers and Baumslag--Solitar subgroups}
\label{sec:BS_criterion}

To any inertial one-relator extension $H*_{\psi}$, we now describe how to associate a unique graph of cyclic groups $\mathcal{G}$. This graph of cyclic groups will be called the graph of cyclic stabilisers of $H*_{\psi}$ and will encode relations between cyclic stabilisers of segments leading out of $H$. Moreover, certain Baumslag--Solitar subgroups of $H*_{\psi}$ can be read off from $\mathcal{G}$.

Denote by
\begin{align*}
\mathcal{A} &= \{[\langle a\rangle]_A \mid A\cap \stab(S)\leqslant \langle a\rangle , S\in \mathcal{S}, \text{ $a$ not a proper power in $A$}\}\\
\mathcal{B} &= \{[\langle b\rangle]_B \mid B\cap \stab(S)\leqslant \langle b\rangle, S\in \mathcal{S}, \text{ $b$ not a proper power in $B$}\}
\end{align*}
In other words, $\mathcal{A}$ (respectively $\mathcal{B}$) is the set of conjugacy classes in $A$ (respectively in $B$) of maximal cyclic subgroups of $A$ (respectively in $B$) which contain a stabiliser $\stab(S)$ for some $S\in \mathcal{S}$.

We now define \emph{the graph of cyclic stabilisers} $\mathcal{G} = (\Gamma, \{\langle c_v\rangle\}, \{\langle c_e\rangle\}, \{\partial_e^{\pm}\})$ as follows. Identify $V(\Gamma)$ with the set $\mathcal{A}\sqcup \mathcal{B}$. Choose any map $\nu:V(\Gamma)\to H$ sending an element $[\langle a\rangle]_A\in \mathcal{A}$ to a generator of a representative $a\in A$ and $[\langle b\rangle]_B\in \mathcal{B}$ to a generator of a representative $b\in B$. There are two types of edges: \emph{$H$-edges} and \emph{$t$-edges}.

For each pair of vertices $v, w\in V(\Gamma)$ and each double coset $\langle \nu(v)\rangle h\langle\nu(w)\rangle$, such that $h\in H$ or $h\in At^{-1}B\sqcup BtA$ and
\[
\langle \nu(v)\rangle\cap \langle \nu(w)\rangle^{h}\neq 1,
\]
there is an edge $e$ connecting $v$ and $w$. We further assume that if $v, w\in \mathcal{A}$ or $v, w\in \mathcal{B}$, then $h\notin\langle \nu(v)\rangle\langle\nu(w)\rangle$. Then the boundary maps $\partial_e^{\pm}$ are induced by the monomorphisms 
\begin{align*}
\langle \nu(v)\rangle\cap \langle \nu(w)\rangle^{h}&\injects \langle \nu(v)\rangle\\
\langle \nu(v)\rangle^{h^{-1}}\cap \langle \nu(w)\rangle&\injects \langle \nu(w)\rangle
\end{align*}
If $h\in H$, then we say that $e$ is an \emph{$H$-edge} and if $h\in At^{-1}B\sqcup BtA$, we say that $e$ is a \emph{$t$-edge}. All edges between vertices in $\mathcal{A}$ and $\mathcal{B}$ are oriented towards $\mathcal{B}$ and a choice of orientation is made for the remaining edges.

\begin{remark}
The isomorphism class of $\mathcal{G}$ as a graph of groups depends only on $\psi$. This follows from the construction and justifies $\mathcal{G}$ being called \emph{the} graph of cyclic stabilisers of $H*_{\psi}$.
\end{remark}

Choose any map $\xi:E(\Gamma)\to G$ such that $\xi(e)\in \langle \nu(o(e))\rangle h\langle\nu(t(e))\rangle$ where $\langle \nu(o(e))\rangle h\langle\nu(t(e))\rangle$ is the double coset associated with $e$. For all $v\in V(\Gamma)$, we may define a group homomorphism 
\[
\mu:\pi_1(\mathcal{G}, v)\to H*_{\psi}
\] 
induced by the choices $\nu, \xi$ as follows.

Each element of $\pi_1(\mathcal{G}, v)$ can be represented by words of the form 
\[
c_{v_0}^{i_0} e^{\epsilon_1}_1 c_{v_1}^{i_1}... e^{\epsilon_n}_n c_{v_n}^{i_n}
\]
where $v_0, v_n = v$, $e_i\in E(\Gamma)$, $\epsilon_i = \pm1$, and $t(e_i) = v_i = o(e_{i+1})$. Then we define
\[
\mu(c_{v_0}^{i_0} e^{\epsilon_1}_1 c_{v_1}^{i_1}... e^{\epsilon_n}_n c_{v_n}^{i_n}) = \nu(v_0)^{i_0}\xi(e_1)^{\epsilon_1}\nu(v_1)^{i_1}...\xi(e_n)^{\epsilon_n}\nu(v_n)^{i_n}
\]
One can check that this is well defined and depends only on $\nu, \xi$. We will call $\mu$ \emph{the homomorphism induced by $\nu, \xi$}.

A path $\gamma:I\to \Gamma$ is \emph{alternating} if it does not traverse two $H$-edges or $t$-edges in a row. An \emph{$H$-path} is a path $\gamma:I\to \Gamma$ that only traverses $H$-edges. We say a word $e^{\epsilon_1}_1 c_{v_1}^{i_1}... e^{\epsilon_n}_n c_{v_n}^{i_n}$ is an \emph{alternating word} if $e^{\epsilon_1}_1*e^{\epsilon_2}_2*...*e^{\epsilon_n}_n$ is an alternating path. Recall that a word $c_{v_0}^{i_0}e^{\epsilon_1}_1 c_{v_1}^{i_1}... e^{\epsilon_n}_n c_{v_n}^{i_n}$ is \emph{cyclically reduced} if all of its cyclic permutations $e^{\epsilon_j}_{j}c_{v_j}^{i_j}... e^{\epsilon_n}_n c_{v_n}^{i_n+i_0}...e^{\epsilon_{j-1}}_{j-1}c_{v_{j-1}}^{i_{j-1}}$ are also reduced. A word is \emph{cyclically alternating} if all of its cyclic permutations are also alternating.

\begin{lemma}
\label{mu_reduced}
If $c = e^{\epsilon_1}_1 c_{v_1}^{i_1}... e^{\epsilon_n}_n c_{v_n}^{i_n}$ is a cyclically reduced and cyclically alternating word, then $\mu(c)$ is a cyclically reduced word. In particular, if $n\geq 2$, then $\langle\mu(c_{v_n}), \mu(c)\rangle$ contains a Baumslag--Solitar subgroup, not conjugate into $H$.
\end{lemma}

\begin{proof}
If $n = 1$ the result is clear, so suppose that $n\geq 2$. We may assume that $e_1$ is a $t$-edge. Taking indices modulo $n$, we have that $\mu(c)$ is not cyclically reduced if and only if $\epsilon_{2j-1} = -\epsilon_{2j+1} = 1$ for some $j$ and
\[
\nu(v_{2j-1})^{i_{2j-1}} \xi(e_{2j})^{\epsilon_{2j}} \nu(v_{2j})^{i_{2j}}\in A
\]
or $\epsilon_{2j-1} = -\epsilon_{2j+1} = -1$ for some $j$ and
\[
\nu(v_{2j-1})^{i_{2j-1}} \xi(e_{2j})^{\epsilon_{2j}} \nu(v_{2j})^{i_{2j}}\in B.
\]
In the first case, $\nu(v_{2j-1})^{i_{2j-1}} \xi(e_{2j})^{\epsilon_{2j}} \nu(v_{2j})^{i_{2j}}\in A$ if and only if $\xi(e_{2j})^{\epsilon_{2j}}\in A$. But this would then imply that $v_{2j-1} = v_{2j}$. Since $A$ is free and $\nu(v_{2j})$ generates a maximal cyclic subgroup, it follows that $\xi(e_{2j})^{\epsilon_{2j}}\in \langle\nu(v_{2j})\rangle$, a contradiction. The same argument is valid for the other case. Thus, when $n\geq 2$, $\mu(c)$ acts hyperbolically on $T$.

Now, by definition of $\mathcal{G}$, there are integers $k, l$, such that
\[
\mu(c)^{-1}\mu(c_{v_n})^k\mu(c) = \mu(c_{v_n})^l.
\]
We may assume that $k$ and $l$ are smallest possible. If $k = \pm 1$ or $l = \pm1$, then $\langle \mu(c), \mu(c_{v_n})\rangle \isom \bs(1, \pm l)$ or $\bs(1, \pm k)$ respectively. Since $\mu(c)$ acts hyperbolically on $T$, it follows that this Baumslag--Solitar subgroup is not conjugate into $H$. If $k, l\neq \pm 1$, since $\mu(c)$ is cyclically reduced, we have that $[\mu(c), \mu(c_{v_n})]$ is also cyclically reduced. Thus, $\langle \mu(c_{v_n})^l, [\mu(c), \mu(c_{v_n})]\rangle\isom \Z^2$. As before, this copy of $\Z^2$ cannot be conjugate into $H$.
\end{proof}

\begin{example}
\label{BS_example}
Consider the following one-relator splitting:
\begin{align*}
\langle a, b, t \mid ta^2t^{-1}bab^{-1}tat^{-1}bab^{-1}\rangle&\isom \langle x, y, z, t \mid txt^{-1} = y, y^2zxz^{-1}yzxz^{-1}\rangle\\
													&\isom \langle x, y, z \mid y^2zxz^{-1}yzxz^{-1}\rangle*_{\psi}
\end{align*}
This one-relator splitting is also a one-relator hierarchy of length one as $y^2zxz^{-1}yzxz^{-1}$ is a primitive element of $F(x, y, z)$.

Since the edge groups of this splitting are cyclic, it follows that $\mathcal{G}$ has two vertices, one corresponding to $[\langle x\rangle]_{\langle x\rangle}$ and the other to $[\langle y\rangle]_{\langle y\rangle}$ and a $t$-edge connecting them. Since $y^2zxz^{-1}yzxz^{-1}$ is a primitive word over $y, zxz^{-1}$, we see that $zx^2z^{-1} = y^{-3}$. Hence there is an edge connecting $[\langle x\rangle]_{\langle x\rangle}$ with $[\langle y\rangle]_{\langle y\rangle}$ corresponding to conjugation by $z$, and where the edge monomorphisms are given by multiplication by $2$ and $-3$ respectively. See Figure \ref{gocg} for the graph of cyclic groups $\mathcal{G}$.

We see that $\pi_1(\mathcal{G})$ has a cyclically alternating and cyclically reduced word $zt^{-1}$ and so our one-relator group contains a Baumslag--Solitar subgroup by Lemma \ref{mu_reduced}. More explicitly, we have
\[
\langle x, zt^{-1} \rangle = \langle a, bt^{-1}\rangle\isom \bs(2, -3).
\]
\end{example}

\begin{figure}
\centering
\includegraphics[scale=1.2]{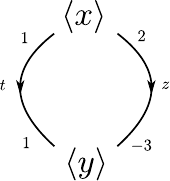}
\caption{The graph of cyclic stabilisers $\mathcal{G}$}
\label{gocg}
\end{figure}

We will denote by $\mathcal{G}_k$ the full subgraph of groups of $\mathcal{G}$ on the vertices corresponding to maximal cyclic subgroups containing some $\stab(S)$ where $S\in \mathcal{S}^i$ for some $i\leq k$.

\begin{lemma}
\label{finite_cyclic_graph}
Let $H*_{\psi}$ be an inertial one-relator extension with $H$ hyperbolic and with both edge groups quasi-convex in $H$. Then the graph of cyclic stabilisers $\mathcal{G}$ is locally finite. Moreover, if $s\Z(\psi)<\infty$, then $\mathcal{G} = \mathcal{G}_{s\Z(\psi)}$ is finite.
\end{lemma}

\begin{proof}
We first show that $\mathcal{G}$ is locally finite. Each vertex has at most one $t$-edge connected to it, so it suffices to only bound the number of $H$-edges connected to any given vertex. Such a bound follows from \cite[Proposition 6.7]{kharlampovich_17} (alternatively, one could use a slight modification of \cite{gitik_98}).

Since $\mathcal{G}$ is locally finite, it follows from the construction that $\mathcal{G}_k$ is finite for all $k$. We claim that $\mathcal{G} = \mathcal{G}_k$ for all $k\geqslant s\Z(\psi)$. Denoting by
\[
\mathcal{S}^{\geqslant n} = \mathcal{S} - \bigcup_{i< n}\mathcal{S}^i
\]
we have that $\stab(S)$ is either trivial or infinite cyclic for all $S\in \mathcal{S}^{\geqslant s\Z(\psi)}$ by Lemma \ref{stab_ranks}. Hence, for all $S\in\mathcal{S}^{\geqslant s\Z(\psi)+1}$, there is some $S'\in \mathcal{S}^{s\Z(\psi)}$ such that $\stab(S)<\stab(S')$. The result follows.
\end{proof}

Lemma \ref{mu_reduced} told us that certain Baumslag--Solitar subgroups of $H*_{\psi}$ could be read off from its graph of cyclic stabilisers. Although we may not find all such subgroups in this way, we now show that under certain conditions, if $H*_{\psi}$ does contain Baumslag--Solitar subgroups, then Lemma \ref{mu_reduced} will always produce a witness.

\begin{theorem}
\label{acylindrical_equivalent}
Let $H*_{\psi}$ be an inertial one-relator extension. Suppose that $s\Z(\psi)<\infty$, that $H$ is hyperbolic and that $A$ and $B$ are quasi-convex in $H$. The following are equivalent:
\begin{enumerate}
\item\label{itm:acylindrical} $G$ acts acylindrically on $T$.
\item\label{itm:BS_subgroup} $G$ does not contain any Baumslag--Solitar subgroups.
\item\label{itm:reduced_word} $\pi_1(\mathcal{G})$ admits no cyclically alternating and cyclically reduced word.
\end{enumerate}
\end{theorem}

\begin{proof}
We prove that (\ref{itm:reduced_word}) implies (\ref{itm:acylindrical}), that (\ref{itm:acylindrical}) implies (\ref{itm:BS_subgroup}) and that (\ref{itm:BS_subgroup}) implies (\ref{itm:reduced_word}) by proving the contrapositive statements.

Let $\mathcal{G}$ be the graph of cyclic stabilisers of $\psi$ and let $\nu:V(\Gamma)\to H$ be a choice of representatives. Suppose that $G$ does not act acylindrically on $T$. Then for any $n$, there exists a sequence of geodesic segments $S_1\subset S_2\subset...\subset S_n$ with $S_i\in \mathcal{S}^i$, each with infinite stabiliser. Let $g_i\in G$ be an element such that the endpoints of $S_i$ are $H$ and $g_iH$. For all $i\geq s\Z(\psi)$, by Lemma \ref{stab_ranks}, $\stab(S_i)$ is cyclic. By definition, each subgroup $\stab(g_i^{-1}\cdot S_i)$ is $H$-conjugate into some $\langle \nu(v)\rangle$ where $v\in V(\Gamma)$. Since $\mathcal{G}$ is finite by Lemma \ref{finite_cyclic_graph}, by the pigeonhole principle there are three integers $s\Z(\psi)<i<j<k$ such that $\stab(g_i^{-1}S_i)$, $\stab(g_j^{-1}S_i)$ and $\stab(g_k^{-1}S_i)$ are $H$-conjugate into some $\langle \nu(v)\rangle$. Thus, there are elements $h_1, h_2, h_3, h_4$ such that:
\begin{align*}
\langle \nu(v)\rangle^{g_jh_2} \cap \langle \nu(v)\rangle^{g_ih_1}&\neq 1,\\
\langle \nu(v)\rangle^{g_kh_4}\cap \langle \nu(v)\rangle^{g_ih_3}&\neq 1.
\end{align*}
In particular, if $f = g_ih_1h_2^{-1}g_j^{-1}$ and $g = g_ih_3h_4^{-1}g_k^{-1}$, then
\begin{align*}
\langle \nu(v)\rangle \cap \langle \nu(v)\rangle^{f}&\neq 1,\\
\langle \nu(v)\rangle \cap \langle \nu(v)\rangle^{g}&\neq 1,\\
\langle \nu(v)\rangle \cap \langle \nu(v)\rangle^{fg}&\neq 1.
\end{align*}
Suppose that both $f$ and $g$ act elliptically on $T$. If they do not both fix a common vertex, then $fg$ acts hyperbolically on $T$. So suppose that they both fix a common vertex $u$. Now the geodesic connecting $u$ with $S_k$ must meet $S_k$ at the midpoint between $g_iH$ and $g_jH$ and the midpoint between $g_iH$ and $g_kH$. Since $j<k$, this is not possible and so we may assume that one of $f$, $g$ or $fg$ acts hyperbolically on $T$. The cyclic reduction of $f$, $g$ or $fg$ provides us with a cyclically alternating and cylically reduced word in $\pi_1(\mathcal{G})$. Thus, (\ref{itm:reduced_word}) implies (\ref{itm:acylindrical}).

Now suppose that $G$ contains a Baumslag--Solitar subgroup $J<G$. Since $H$ is hyperbolic and so cannot contain a Baumslag--Solitar subgroup, $J$ cannot act elliptically on $T$. It follows from \cite[Theorem 2.1]{minasyan_15} that $J$ cannot act acylindrically on $T$. Hence, $G$ does not act acylindrically on $T$ and (\ref{itm:acylindrical}) implies (\ref{itm:BS_subgroup}).

Finally, Lemma \ref{mu_reduced} shows that (\ref{itm:BS_subgroup}) implies (\ref{itm:reduced_word}).
\end{proof}

\begin{example}
\label{no_BS_example}
Consider Example \ref{stable_example} from the introduction:
\[
H*_{\psi} = \langle x, y, z \mid z^2yz^2x^{-2} \rangle*_{\psi}
\]
where $\psi:A\to B$ is given by $\psi(x) = y$, $\psi(y) = z$. We showed that $H*_{\psi}$ is $\Z$-stable and that the vertex group was freely generated by $x, z$. We see that the edge groups are generated by $x, z^{-2}x^2z^{-2}$ and $x^2, z$ respectively. Every subgroup of the form $A\cap B^g$, $A\cap A^g$ or $B\cap B^g$ is either trivial, or conjugate to $A$, $B$ or one of the following:
\begin{align*}
\langle x, z^{-2}x^2z^{-2}\rangle\cap \langle x^2, z\rangle &= \langle x^2, z^{-2}x^2z^{-2}\rangle,\\
\langle x, z^{-2}x^2z^{-2}\rangle\cap \langle x^2, z\rangle^x &= \langle x^2\rangle,\\
\langle x^2, z\rangle\cap \langle x^2, z\rangle^x & = \langle x^2\rangle.
\end{align*}
By applying $\psi$ to $x$ and $\psi^{-1}$ to $x^2 = z^2yz^2$, we see that the vertices in $\Gamma$ corresponding to stabilisers of elements in $\mathcal{S}^2$ are the $\mathcal{B}$-vertex $[\langle y\rangle]_{B}$ and the $\mathcal{A}$-vertices $[\langle y\rangle]_{A}$ and $[\langle y^2xy^2\rangle]_{A}$. The graph of cyclic stabilisers can be seen in Figure \ref{gocg_no_BS}. Since $\mathcal{G}$ does not admit any cyclically alternating and cyclically reduced words, by Theorem \ref{acylindrical_equivalent} we see that $H*_{\psi}$ does not contain any Baumslag--Solitar subgroups.
\end{example}

\begin{figure}
\centering
\includegraphics[scale=1.2]{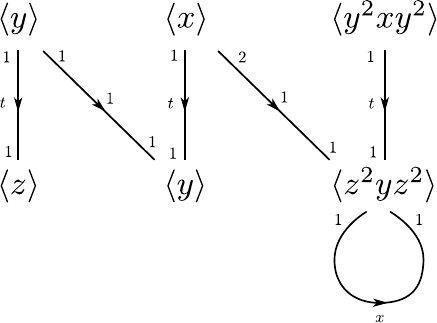}
\caption{The graph of cyclic stabilisers $\mathcal{G}$}
\label{gocg_no_BS}
\end{figure}

\begin{remark}
Algorithmic problems relating to this section are handled in detail in the author's thesis \cite{my_thesis}. Combining work of Howie \cite{howie_05}, Stallings \cite{sta_83} and Kharlampovich--Miasnikov--Weil \cite{kharlampovich_17}, there it is shown that given as input an inertial one-relator extension $H*_{\psi}$ with the assumptions of Theorem \ref{acylindrical_equivalent}, the $\Z$-stable number and the graph of cyclic stabilisers are computable. Thus, one can effectively decide whether $H*_{\psi}$ contains a Baumslag--Solitar subgroup or not. The problem of deciding whether an arbitrary one-relator group contains a Baumslag--Solitar subgroup is still open.
\end{remark}

%%%%%%%%%%%%%%%%%%
\section{Main results and further questions}
\label{sec:main}

We are finally ready to prove our main results. A one-relator tower (respectively, hierarchy) $X_N\immerses ...\immerses X_1\immerses X_0$ is an \emph{acylindrical tower} (respectively, hierarchy) if $\pi_1(X_i)$ acts acylindrically on the Bass-Serre tree associated with the one-relator splitting $\pi_1(X_i) = \pi_1(X_{i+1})*_{\psi_i}$ for all $i$. It is a \emph{$\Z$-stable tower} (respectively, hierarchy) if $s\Z(\psi_i)<\infty$ for all $i$.

\begin{theorem}
\label{main}
Let $X$ be a one-relator complex and $X_N\immerses ...\immerses X_1\immerses X_0 = X$ a one-relator hierarchy. The following are equivalent:
\begin{enumerate}
\item\label{itm:hqch} $X_N\immerses ...\immerses X_1\immerses X_0 = X$ is a quasi-convex hierarchy and $\pi_1(X)$ is hyperbolic.
\item\label{itm:acylindrical_h} $X_N\immerses ...\immerses X_1\immerses X_0 = X$ is an acylindrical hierarchy.
\item\label{itm:bs} $X_N\immerses ...\immerses X_1\immerses X_0 = X$ is a $\Z$-stable hierarchy and $\pi_1(X)$ contains no Baumslag--Solitar subgroups.
\end{enumerate}
Moreover, if any of the above is satisfied, then $\pi_1(X)$ is virtually special and the image of $\pi_1(A)$ in $\pi_1(X)$ is quasi-convex for any connected subcomplex $A\subset X_i$.
\end{theorem}

\begin{proof}
The proof is by induction. The base case is clear. By the inductive hypothesis, we may assume that $\pi_1(X_1)$ is hyperbolic, virtually special and that $\pi_1(A)<\pi_1(X_1)$ is quasi-convex for any connected full subcomplex $A\subset X_i$ with $i\geqslant 1$. The equivalence between (\ref{itm:hqch}) and (\ref{itm:acylindrical_h}) now follows from Theorem \ref{kapo_main}. The equivalence between (\ref{itm:acylindrical_h}) and (\ref{itm:bs}) is Theorem \ref{acylindrical_equivalent}.

Finally, assuming $\pi_1(X)$ is hyperbolic and the hierarchy is quasi-convex, $\pi_1(X)$ is virtually special by \cite[Theorem 13.3]{wise_21_quasiconvex} and the image of $\pi_1(A)$ in $\pi_1(X)$ is quasi-convex for any connected subcomplex $A\subset X_i$ by Theorem \ref{one-relator_normal_form_inductive} and induction.
\end{proof}

We now prove a stronger form of \cite[Conjecture 1.9]{louder_21}.

\begin{theorem}
\label{hyperbolicity_theorem}
Let $X$ be a one-relator complex with negative immersions. Then $\pi_1(X)$ is hyperbolic, virtually special and all of its one-relator hierarchies $X_N\immerses...\immerses X_1\immerses X_0 = X$ are quasi-convex hierarchies.
\end{theorem}

\begin{proof}
Since $\pi_1(X)$ does not contain any Baumslag--Solitar subgroups \cite[Corollary 1.8]{louder_21}, the result follows from Corollary \ref{stable_negative} and Theorem \ref{main}.
\end{proof}

\begin{corollary}
\label{isomorphism}
The isomorphism problem for one-relator groups with negative immersions is decidable within the class of one-relator groups.
\end{corollary}

\begin{proof}
Let $G$ be a one-relator group with negative immersions and let $H$ be any one-relator group. By \cite[Theorem 1.3 \& Lemma 6.4]{louder_21}, there is an algorithm to decide whether $H$ has negative immersions or not. If it does not, then it is not isomorphic to $G$. If it does, then both $G$ and $H$ are hyperbolic by Theorem \ref{hyperbolicity_theorem} and so isomorphism between the two can be decided by \cite{dahmani_08}.
\end{proof}

We may now solve a problem of Baumslag's \cite[Problem 4]{baumslag_85}.

\begin{corollary}
\label{parafree}
Parafree one-relator groups have negative immersions. In particular, they are hyperbolic, virtually special and their isomorphism problem is decidable.
\end{corollary}

\begin{proof}
By \cite[Theorem 4.2]{baumslag_69_parafree} and \cite[Theorems 1.3 \& 1.5]{louder_21}, parafree one-relator groups have negative immersions. Now the result follows from Theorems \ref{hyperbolicity_theorem} and \ref{isomorphism}.
\end{proof}

\begin{corollary}
\label{residual_finiteness}
One-relator groups with negative immersions are residually finite and linear.
\end{corollary}

\begin{corollary}
\label{hyperbolicity_corollary}
Let $X$ be a one-relator complex with negative immersions. Then every finitely generated subgroup of $\pi_1(X)$ is hyperbolic.
\end{corollary}

\begin{proof}
Follows from \cite[Corollary 7.8]{gersten_96} and \cite[Theorem A]{louder_21_uniform}.
\end{proof}

Extending any of these results to all one-relator groups would require a better understanding of one-relator hierarchies that are not $\Z$-stable. Therefore, we ask the following question.

\begin{problem}
Characterise non $\Z$-stable one-relator hierarchies.
\end{problem}

Any one-relator group satisfying the hypothesis of Brown's criterion (see \cite{brown_87}), either has a one-relator splitting that is not $\Z$-stable, or is isomorphic to a Baumslag--Solitar group $\bs(1, k)$. More generally, if $\pi_1(X)\isom \pi_1(X_1)*_{\psi}$ is a one-relator splitting and there is some $g\in \pi_1(X)$ acting hyperbolically on its Bass--Serre tree and some $A_n\in [A_n]\in\bar{\mathcal{A}}_n^{\psi}$ such that $A^g_n<A_n$, then $\langle A_n, g\rangle$ splits as an ascending HNN-extension of a finitely generated free group. We ask whether this is the only situation that can occur.

\begin{question}
Let $X$ be a one-relator complex and $\pi_1(X)\isom \pi_1(X_1)*_{\psi}$ a one-relator splitting that is not $\Z$-stable. Is there some $n\geq 0$ such that one of the following holds 
\begin{align*}
\bar{\mathcal{A}}_n^{\psi} &= \bar{\mathcal{A}}_{n+i}^{\psi}\\
\bar{\mathcal{A}}_n^{\psi^{-1}} &= \bar{\mathcal{A}}_{n+i}^{\psi^{-1}}
\end{align*}
for all $i\geq 0$? 

If so, does there exist an immersion of one-relator complexes $Z\immerses X$ that does not factor through $X_1\immerses X$ and such that $\pi_1(Z)$ has non-empty BNS invariant?
\end{question}

\bibliographystyle{amsalpha}
\bibliography{bibliography}

\newcommand{\etalchar}[1]{$^{#1}$}
\providecommand{\bysame}{\leavevmode\hbox to3em{\hrulefill}\thinspace}
\providecommand{\MR}{\relax\ifhmode\unskip\space\fi MR }
% \MRhref is called by the amsart/book/proc definition of \MR.
\providecommand{\MRhref}[2]{%
  \href{http://www.ams.org/mathscinet-getitem?mr=#1}{#2}
}
\providecommand{\href}[2]{#2}
\begin{thebibliography}{FRMW21}

\bibitem[ABC{\etalchar{+}}91]{alonso_91}
J.~M. Alonso, T.~Brady, D.~Cooper, V.~Ferlini, M.~Lustig, M.~Mihalik,
  M.~Shapiro, and H.~Short, \emph{Notes on word hyperbolic groups}, Group
  theory from a geometrical viewpoint ({T}rieste, 1990), World Sci. Publ.,
  River Edge, NJ, 1991, Edited by Short, pp.~3--63. \MR{1170363}

\bibitem[AG99]{allcock_99}
D.~J. Allcock and S.~M. Gersten, \emph{A homological characterization of
  hyperbolic groups}, Invent. Math. \textbf{135} (1999), no.~3, 723--742.
  \MR{1669272}

\bibitem[Bau67]{baumslag_67}
Gilbert Baumslag, \emph{Residually finite one-relator groups}, Bull. Amer.
  Math. Soc. \textbf{73} (1967), 618--620. \MR{212078}

\bibitem[Bau69]{baumslag_69_parafree}
\bysame, \emph{Groups with the same lower central sequence as a relatively free
  group. {II}. {P}roperties}, Trans. Amer. Math. Soc. \textbf{142} (1969),
  507--538. \MR{245653}

\bibitem[Bau86]{baumslag_85}
\bysame, \emph{A survey of groups with a single defining relation}, Proceedings
  of groups---{S}t. {A}ndrews 1985, London Math. Soc. Lecture Note Ser., vol.
  121, Cambridge Univ. Press, Cambridge, 1986, pp.~30--58. \MR{896500}

\bibitem[BC06]{baumslag_06}
Gilbert Baumslag and Sean Cleary, \emph{Parafree one-relator groups}, J. Group
  Theory \textbf{9} (2006), no.~2, 191--201. \MR{2220574}

\bibitem[BCH04]{baumslag_04}
Gilbert Baumslag, Sean Cleary, and George Havas, \emph{Experimenting with
  infinite groups. {I}}, Experiment. Math. \textbf{13} (2004), no.~4, 495--502.
  \MR{2118274}

\bibitem[Bes04]{bestvina_04}
Mladen Bestvina, \emph{Questions in geometric group theory}, 2004,
  https://www.math.utah.edu/~bestvina/eprints/questions-updated.pdf.

\bibitem[BF92]{bestvina_92_combination}
Mladen Bestvina and Mark Feighn, \emph{A combination theorem for negatively
  curved groups}, J. Differential Geom. \textbf{35} (1992), no.~1, 85--101.
  \MR{1152226}

\bibitem[BFR19]{baumslag_19}
Gilbert Baumslag, Benjamin Fine, and Gerhard Rosenberger, \emph{One-relator
  groups: an overview}, Groups {S}t {A}ndrews 2017 in {B}irmingham, London
  Math. Soc. Lecture Note Ser., vol. 455, Cambridge Univ. Press, Cambridge,
  2019, pp.~119--157. \MR{3931411}

\bibitem[BM22]{blufstein_19}
Mart\'{\i}n~Axel Blufstein and El\'{\i}as Minian, \emph{Strictly systolic
  angled complexes and hyperbolicity of one-relator groups}, Algebr. Geom.
  Topol. \textbf{22} (2022), no.~3, 1159--1175. \MR{4474781}

\bibitem[Bra99]{brady_99}
Noel Brady, \emph{Branched coverings of cubical complexes and subgroups of
  hyperbolic groups}, J. London Math. Soc. (2) \textbf{60} (1999), no.~2,
  461--480. \MR{1724853}

\bibitem[Bro87]{brown_87}
Kenneth~S. Brown, \emph{Trees, valuations, and the {B}ieri-{N}eumann-{S}trebel
  invariant}, Invent. Math. \textbf{90} (1987), no.~3, 479--504. \MR{914847}

\bibitem[BW13]{BW13}
Hadi Bigdely and Daniel~T. Wise, \emph{Quasiconvexity and relatively hyperbolic
  groups that split}, Michigan Math. J. \textbf{62} (2013), no.~2, 387--406.
  \MR{3079269}

\bibitem[CH21]{cashen_21}
Christopher~H. Cashen and Charlotte Hoffmann, \emph{{Short, Highly Imprimitive
  Words Yield Hyperbolic One-Relator Groups}}, Experimental Mathematics
  \textbf{0} (2021), no.~0, 1--10.

\bibitem[Che21]{chen_21}
Haimiao Chen, \emph{Solving the isomorphism problems for two families of
  parafree groups}, J. Algebra \textbf{585} (2021), 616--636. \MR{4282650}

\bibitem[CM82]{chandler_82}
Bruce Chandler and Wilhelm Magnus, \emph{The history of combinatorial group
  theory}, Studies in the History of Mathematics and Physical Sciences, vol.~9,
  Springer-Verlag, New York, 1982, A case study in the history of ideas.
  \MR{680777}

\bibitem[Col04]{collins_04}
Donald~J. Collins, \emph{Intersections of {M}agnus subgroups of one-relator
  groups}, Groups: topological, combinatorial and arithmetic aspects, London
  Math. Soc. Lecture Note Ser., vol. 311, Cambridge Univ. Press, Cambridge,
  2004, pp.~255--296. \MR{2073350}

\bibitem[Col08]{collins_08}
\bysame, \emph{Intersections of conjugates of {M}agnus subgroups of one-relator
  groups}, The {Z}ieschang {G}edenkschrift, Geom. Topol. Monogr., vol.~14,
  Geom. Topol. Publ., Coventry, 2008, pp.~135--171. \MR{2484702}

\bibitem[Dah03]{Da03}
Fran\c{c}ois Dahmani, \emph{Combination of convergence groups}, Geom. Topol.
  \textbf{7} (2003), 933--963. \MR{2026551}

\bibitem[DG08]{dahmani_08}
Fran\c{c}ois Dahmani and Daniel Groves, \emph{The isomorphism problem for toral
  relatively hyperbolic groups}, Publ. Math. Inst. Hautes \'{E}tudes Sci.
  (2008), no.~107, 211--290. \MR{2434694}

\bibitem[DV96]{dicks_96}
Warren Dicks and Enric Ventura, \emph{The group fixed by a family of injective
  endomorphisms of a free group}, Contemporary Mathematics, vol. 195, American
  Mathematical Society, Providence, RI, 1996. \MR{1385923}

\bibitem[FRMW21]{fine_21}
Benjamin Fine, Gerhard Rosenberger, Anja Moldenhauer, and Leonard Wienke,
  \emph{Topics in infinite group theory---{N}ielsen methods, covering spaces,
  and hyperbolic groups}, De Gruyter STEM, De Gruyter, Berlin, [2021]
  \copyright 2021. \MR{4368862}

\bibitem[FRS97]{fine_97}
Benjamin Fine, Gerhard Rosenberger, and Michael Stille, \emph{The isomorphism
  problem for a class of para-free groups}, Proc. Edinburgh Math. Soc. (2)
  \textbf{40} (1997), no.~3, 541--549. \MR{1475915}

\bibitem[Ger92]{gersten_92}
Stephen~M. Gersten, \emph{Problems on automatic groups}, Algorithms and
  classification in combinatorial group theory ({B}erkeley, {CA}, 1989), Math.
  Sci. Res. Inst. Publ., vol.~23, Springer, New York, 1992, pp.~225--232.
  \MR{1230636}

\bibitem[Ger96]{gersten_96}
\bysame, \emph{Subgroups of word hyperbolic groups in dimension {$2$}}, J.
  London Math. Soc. (2) \textbf{54} (1996), no.~2, 261--283. \MR{1405055}

\bibitem[GKL21]{gardam_21}
Giles Gardam, Dawid Kielak, and Alan~D. Logan, \emph{Algebraically hyperbolic
  groups}, 2021, arXiv:2112.01331.

\bibitem[GMRS98]{gitik_98}
Rita Gitik, Mahan Mitra, Eliyahu Rips, and Michah Sageev, \emph{Widths of
  subgroups}, Transactions of the American Mathematical Society \textbf{350}
  (1998), no.~1, 321--329.

\bibitem[Gro87]{gromov_87}
Mikhael Gromov, \emph{Hyperbolic groups}, Essays in group theory, Math. Sci.
  Res. Inst. Publ., vol.~8, Springer, New York, 1987, pp.~75--263. \MR{919829}

\bibitem[Gro93]{gromov_93}
M.~Gromov, \emph{Asymptotic invariants of infinite groups}, Geometric group
  theory, {V}ol. 2 ({S}ussex, 1991), London Math. Soc. Lecture Note Ser., vol.
  182, Cambridge Univ. Press, Cambridge, 1993, pp.~1--295. \MR{1253544}

\bibitem[GW19]{gardam_18}
Giles Gardam and Daniel~J. Woodhouse, \emph{The geometry of one-relator groups
  satisfying a polynomial isoperimetric inequality}, Proc. Amer. Math. Soc.
  \textbf{147} (2019), no.~1, 125--129. \MR{3876736}

\bibitem[Hat02]{hatcher_00}
Allen Hatcher, \emph{Algebraic topology}, Cambridge University Press,
  Cambridge, 2002. \MR{1867354}

\bibitem[HK17]{hung_17}
Do~Viet Hung and Vu~The Khoi, \emph{Applications of the {A}lexander ideals to
  the isomorphism problem for families of groups}, Proc. Edinb. Math. Soc. (2)
  \textbf{60} (2017), no.~1, 177--185. \MR{3589847}

\bibitem[HK20]{hung_20}
\bysame, \emph{Twisted {A}lexander ideals and the isomorphism problem for a
  family of parafree groups}, Proc. Edinb. Math. Soc. (2) \textbf{63} (2020),
  no.~3, 780--806. \MR{4163871}

\bibitem[How81]{howie_81}
James Howie, \emph{On pairs of {$2$}-complexes and systems of equations over
  groups}, J. Reine Angew. Math. \textbf{324} (1981), 165--174. \MR{614523}

\bibitem[How87]{howie_87}
\bysame, \emph{How to generalize one-relator group theory}, Combinatorial group
  theory and topology ({A}lta, {U}tah, 1984), Ann. of Math. Stud., vol. 111,
  Princeton Univ. Press, Princeton, NJ, 1987, pp.~53--78. \MR{895609}

\bibitem[How05]{howie_05}
\bysame, \emph{Magnus intersections in one-relator products}, Michigan Math. J.
  \textbf{53} (2005), no.~3, 597--623. \MR{2207211}

\bibitem[HW01]{hruska_01}
Geoffrey~Christopher Hruska and Daniel~T. Wise, \emph{Towers, ladders and the
  {B}. {B}. {N}ewman spelling theorem}, J. Aust. Math. Soc. \textbf{71} (2001),
  no.~1, 53--69. \MR{1840493}

\bibitem[HW16]{helfer_16}
Joseph Helfer and Daniel~T. Wise, \emph{Counting cycles in labeled graphs: the
  nonpositive immersion property for one-relator groups}, Int. Math. Res. Not.
  IMRN (2016), no.~9, 2813--2827. \MR{3519130}

\bibitem[IMM23]{italiano_21}
Giovanni Italiano, Bruno Martelli, and Matteo Migliorini, \emph{Hyperbolic
  5-manifolds that fiber over {$S^1$}}, Invent. Math. \textbf{231} (2023),
  no.~1, 1--38. \MR{4526820}

\bibitem[IS98]{ivanov_98}
Sergei~V. Ivanov and Paul~E. Schupp, \emph{On the hyperbolicity of small
  cancellation groups and one-relator groups}, Trans. Amer. Math. Soc.
  \textbf{350} (1998), no.~5, 1851--1894. \MR{1401522}

\bibitem[Iva18]{ivanov_18}
Sergei~V. Ivanov, \emph{The intersection of subgroups in free groups and linear
  programming}, Math. Ann. \textbf{370} (2018), no.~3-4, 1909--1940.
  \MR{3770185}

\bibitem[Kap00]{kapovich_00}
Ilya Kapovich, \emph{Mapping tori of endomorphisms of free groups}, Comm.
  Algebra \textbf{28} (2000), no.~6, 2895--2917. \MR{1757436}

\bibitem[Kap01]{kapovich_01}
\bysame, \emph{The combination theorem and quasiconvexity}, Internat. J.
  Algebra Comput. \textbf{11} (2001), no.~2, 185--216. \MR{1829050}

\bibitem[KMS60]{karrass_60}
Abraham Karrass, Wilhelm Magnus, and Donald Solitar, \emph{Elements of finite
  order in groups with a single defining relation}, Comm. Pure Appl. Math.
  \textbf{13} (1960), 57--66. \MR{124384}

\bibitem[KMW17]{kharlampovich_17}
Olga Kharlampovich, Alexei Miasnikov, and Pascal Weil, \emph{Stallings graphs
  for quasi-convex subgroups}, J. Algebra \textbf{488} (2017), 442--483.
  \MR{3680926}

\bibitem[Lin22]{my_thesis}
Marco Linton, \emph{One-relator hierarchies}, Ph.D. thesis, University of
  Warwick, UK, 2022.

\bibitem[LL94]{lewis_94}
Robert~H. Lewis and Sal Liriano, \emph{Isomorphism classes and derived series
  of certain almost-free groups}, Experiment. Math. \textbf{3} (1994), no.~3,
  255--258. \MR{1329373}

\bibitem[LS01]{lyn_00}
Roger~C. Lyndon and Paul~E. Schupp, \emph{Combinatorial group theory}, Classics
  in Mathematics, Springer-Verlag, Berlin, 2001, Reprint of the 1977 edition.
  \MR{1812024}

\bibitem[LW17]{louder_17}
Larsen Louder and Henry Wilton, \emph{Stackings and the {$W$}-cycles
  conjecture}, Canad. Math. Bull. \textbf{60} (2017), no.~3, 604--612.
  \MR{3679733}

\bibitem[LW22]{louder_21}
\bysame, \emph{Negative immersions for one-relator groups}, Duke Math. J.
  \textbf{171} (2022), no.~3, 547--594. \MR{4382976}

\bibitem[LW24]{louder_21_uniform}
\bysame, \emph{Uniform negative immersions and the coherence of one-relator
  groups}, Invent. Math. \textbf{236} (2024), no.~2, 673--712. \MR{4728240}

\bibitem[Lyn50]{lyndon_50}
Roger~C. Lyndon, \emph{Cohomology theory of groups with a single defining
  relation}, Ann. of Math. (2) \textbf{52} (1950), 650--665. \MR{47046}

\bibitem[Mag30]{magnus_30}
Wilhelm Magnus, \emph{\"{U}ber diskontinuierliche {G}ruppen mit einer
  definierenden {R}elation. ({D}er {F}reiheitssatz)}, J. Reine Angew. Math.
  \textbf{163} (1930), 141--165. \MR{1581238}

\bibitem[Mas06]{masters_06}
Joseph~D. Masters, \emph{Heegaard splittings and 1-relator groups}, unpublished
  paper, 2006.

\bibitem[MKS66]{magnus_04}
Wilhelm Magnus, Abraham Karrass, and Donald Solitar, \emph{Combinatorial group
  theory: {P}resentations of groups in terms of generators and relations},
  Interscience Publishers [John Wiley \& Sons], New York-London-Sydney, 1966.
  \MR{0207802}

\bibitem[MO15]{minasyan_15}
Ashot Minasyan and Denis Osin, \emph{Acylindrical hyperbolicity of groups
  acting on trees}, Math. Ann. \textbf{362} (2015), no.~3-4, 1055--1105.
  \MR{3368093}

\bibitem[Mol67]{moldavanskii_67}
D.~I. Moldavanski\u{\i}, \emph{Certain subgroups of groups with one defining
  relation}, Sibirsk. Mat. \v{Z}. \textbf{8} (1967), 1370--1384. \MR{0220810}

\bibitem[Mut21a]{mutanguha_21_stable}
Jean~Pierre Mutanguha, \emph{Constructing stable images}, preprint on webpage
  at mutanguha.com/pdfs/relimmalgo.pdf, 2021.

\bibitem[Mut21b]{mutanguha_21}
Jean~Pierre Mutanguha, \emph{The dynamics and geometry of free group
  endomorphisms}, Adv. Math. \textbf{384} (2021), Paper No. 107714, 60.
  \MR{4237417}

\bibitem[MUW11]{myasnikov_11}
Alexei Myasnikov, Alexander Ushakov, and Dong~Wook Won, \emph{The word problem
  in the {B}aumslag group with a non-elementary {D}ehn function is polynomial
  time decidable}, J. Algebra \textbf{345} (2011), 324--342. \MR{2842068}

\bibitem[New68]{newman_68}
B.~B. Newman, \emph{Some results on one-relator groups}, Bull. Amer. Math. Soc.
  \textbf{74} (1968), 568--571. \MR{222152}

\bibitem[Ros13]{rosenmann_13}
Amnon Rosenmann, \emph{On the intersection of subgroups in free groups: echelon
  subgroups are inert}, Groups Complex. Cryptol. \textbf{5} (2013), no.~2,
  211--221. \MR{3245107}

\bibitem[Sel97]{sela_97}
Z.~Sela, \emph{Acylindrical accessibility for groups}, Invent. Math.
  \textbf{129} (1997), no.~3, 527--565. \MR{1465334}

\bibitem[Ser03]{serre_80}
Jean-Pierre Serre, \emph{Trees}, Springer Monographs in Mathematics,
  Springer-Verlag, Berlin, 2003, Translated from the French original by John
  Stillwell, Corrected 2nd printing of the 1980 English translation.
  \MR{1954121}

\bibitem[Sta83]{sta_83}
John~R. Stallings, \emph{Topology of finite graphs}, Invent. Math. \textbf{71}
  (1983), no.~3, 551--565. \MR{695906}

\bibitem[Tar92]{tardos_92}
G\'{a}bor Tardos, \emph{On the intersection of subgroups of a free group},
  Invent. Math. \textbf{108} (1992), no.~1, 29--36. \MR{1156384}

\bibitem[Wis03]{wise_03}
Daniel~T. Wise, \emph{Nonpositive immersions, sectional curvature, and subgroup
  properties}, Electron. Res. Announc. Amer. Math. Soc. \textbf{9} (2003),
  1--9. \MR{1988866}

\bibitem[Wis04]{wise_04}
D.~T. Wise, \emph{Sectional curvature, compact cores, and local
  quasiconvexity}, Geom. Funct. Anal. \textbf{14} (2004), no.~2, 433--468.
  \MR{2062762}

\bibitem[Wis05]{wise_05_cat}
Daniel~T. Wise, \emph{Approximating flats by periodic flats in {CAT}(0) square
  complexes}, Canad. J. Math. \textbf{57} (2005), no.~2, 416--448. \MR{2124924}

\bibitem[Wis14]{wise_14}
\bysame, \emph{Cubular tubular groups}, Trans. Amer. Math. Soc. \textbf{366}
  (2014), no.~10, 5503--5521. \MR{3240932}

\bibitem[Wis20]{wise_21}
\bysame, \emph{An invitation to coherent groups}, What's next?---the
  mathematical legacy of {W}illiam {P}. {T}hurston, Ann. of Math. Stud., vol.
  205, Princeton Univ. Press, Princeton, NJ, 2020, pp.~326--414. \MR{4205645}

\bibitem[Wis21]{wise_21_quasiconvex}
\bysame, \emph{The structure of groups with a quasiconvex hierarchy}, Annals of
  Mathematics Studies, vol. 209, Princeton University Press, Princeton, NJ,
  [2021] \copyright 2021. \MR{4298722}

\end{thebibliography}

\end{document}